\documentclass[11pt]{amsart} 
\usepackage{latexsym}\usepackage{amsmath}\usepackage{xspace}\usepackage{enumerate}
\usepackage{amsfonts}\usepackage{amscd}\usepackage{syntonly}\usepackage{amssymb}
\usepackage{hyperref}
\usepackage{bbm}
\usepackage{graphicx}
\usepackage{amssymb}
\usepackage{mathrsfs}
\usepackage{mathtools}
\usepackage{latexsym} 
\usepackage{color}  
\usepackage[alwaysadjust]{paralist}

\newtheorem{thm}{Theorem}
\newtheorem{lemma}[thm]{Lemma}
\newtheorem{prop}[thm]{Proposition}

\newtheorem{rem} [thm]{Remark}

\def\Si{\Sigma}

\renewcommand\phi{\varphi}

\def\tr{\textrm{tr}}

\def\id{\textrm{id}}

\newcommand{\Z}{\mathbb{Z}}

\newcommand{\R}{\mathbb{R}}

\newcommand{\BbD}{\mathbb D}
\newcommand{\BPi}{\overline{\Pi}}
\newcommand{\TPi}{\widetilde{\Pi}}
\newcommand{\Mg}{\calM_\gamma}
\newcommand{\Mgc}{\overline{\Mg}}
\newcommand{\Mgt}{\widetilde{\Mg}}
\newcommand{\Mgone}{\calM_{\gamma,1}}
\newcommand{\Mgonec}{\overline{\calM_{\gamma,1}}}
\newcommand{\Mgonet}{\widetilde{\calM_{\gamma,1}}}

\newcommand{\DD}{\mathbb D}

\newcommand{\RR}{\mathbb R}

\newcommand{\del}{\partial}

\newcommand{\Met}{\tt{Met}}
\newcommand{\Diff}{\mbox{Diff\,}}

\newcommand{\calC}{{\mathcal C}}

\newcommand{\calM}{{\mathcal M}}

\newcommand{\calO}{{\mathcal O}}

\newcommand{\calS}{{\mathcal S}}
\newcommand{\calT}{{\mathcal T}}
\newcommand{\calU}{{\mathcal U}}
\newcommand{\calV}{{\mathcal V}}
\newcommand{\calW}{{\mathcal W}}

\newcommand{\wtM}{\widetilde{M}}

\def\WPe{Weil-Petersson }
\newcommand{\WP}{\mathrm{WP}}

\newcommand{\reg}{{\mathrm{reg}}}
\newcommand{\fc}{\mathfrak c}
\newcommand{\ver}{\mathrm{ver}}
\newcommand{\Stt}{\mathcal S_{\operatorname{tt}}}

\newcommand{\wh}{\widehat}

\title{}

\title{Asymptotics of the Weil--Petersson metric}
\author[R.~Mazzeo]{Rafe Mazzeo}
\address[R.~Mazzeo]{Department of Mathematics, Stanford University}
\email{mazzeo@math.stanford.edu}
\urladdr{http://web.stanford.edu/~rmazzeo/cgi-bin/}

\author[J.~Swoboda]{Jan Swoboda}
\address[J.~Swoboda]{Mathematisches Institut der LMU M\"unchen}
\email{swoboda@mathematik.uni-muenchen.de}
\urladdr{http://www.mathematik.uni-muenchen.de/~swoboda/}

\date{\today}


\begin{document}
\maketitle

 \begin{abstract}  
We consider the Riemann moduli space $\mathcal M_{\gamma}$ of conformal structures on a compact surface of 
genus $\gamma>1$ together with its Weil-Petersson metric $g_{\WP}$. Our main result is that $g_{\WP}$ admits a complete
polyhomogeneous expansion in powers of the lengths of the short geodesics up to the singular divisors of the 
Deligne-Mumford compactification of $\mathcal M_{\gamma}$.
 \end{abstract}
\section{Introduction}
The Riemann moduli space $\calM_\gamma$ of conformal structures on a compact surface of genus $\gamma > 1$
is an object of key importance in several branches of mathematics and mathematical physics. Part of its fascination
is that it is endowed with numerous natural geometric structures. We focus here on one of these, the \WPe metric $g_{\WP}$. 
We recall that $(\calM_\gamma, g_{\WP})$ is an incomplete Riemannian space, and a quasi-projective variety of (complex) 
dimension $3\gamma-3$. Its Deligne-Mumford compactification contains a collection of immersed divisors 
$D_1 \cup \ldots \cup D_N$, $N = 2\gamma-2$, which meet with simple normal crossings, and along which 
$\overline{\calM}_\gamma$ has orbifold singularities, cf.\ \cite{JW}.  We denote by $\BbD$ this entire
divisor, i.e., the union of the $D_j$. The \WPe metric is not fully compatible with this 
compactification in the sense that the local asymptotic behaviour of $g_{\WP}$ near these divisors is somewhat
complicated: normal to each divisor it has cusp-like behavior, but at intersections of the divisors, these normal cusps 
do not interact. Our goal in this paper is to sharpen the work of Masur \cite{Masur}, Yamada \cite{Yamada} 
and Wolpert \cite{Wolpert,Wolpert1}, and in a slightly different direction, Liu-Sun-Yau \cite{LSY1, LSY2}, each 
of whom provided successively finer estimates. This work also refines
Wolpert's very recent paper \cite{Wolpert2}, which proves a certain uniformity of derivatives for this metric.
We prove here that $g_{\WP}$ has a complete polyhomogeneous asymptotic expansion at $\BbD$, with product 
type expansions at the intersections of the $D_j$, see Theorem \ref{thm:mainthm} below for the precise statement. 
In general terms, a polyhomogeneous expansion is an asymptotic series, which remains valid even after differentiation, but which 
may involve possibly fractional powers of the distance function to the boundary.  We obtain, in fact, that the metric
coefficients are essentially `log-smooth', so in other words, involve only nonnegative integer powers of a natural
boundary defining function $\rho$, where each term $\rho^k$ is multiplied by a polynomial in $\log \rho$.  (In fact,
$\rho$ is the square root of the length of the degenerating geodesic.) 
As we explain later, this is the sharpest type of regularity one might hope to obtain, and in particular provides
more information than the `stable regularity' estimates of \cite{Wolpert2}.  What makes this somewhat different
than analogous regularity results, cf.\ \cite{LM}, \cite{JMR}, \cite{CMR}, is that $g_{\WP}$ does not satisfy an
elliptic equation, but instead is the induced $L^2$ metric in the gauge-theoretic construction of $\calM_\gamma$ -- or
in other words, is the restriction of the $L^2$ metric on the space of all symmetric $2$-tensors to the finite dimensional
subspace of transverse-traceless tensors. Our work leaves open the precise identification of the terms in the expansion 
of $g_{\WP}$. That is a very different sort of task, but one which becomes possible only once one has established that
an expansion exists at all! The first few terms are computed in \cite{OW}. The result here is also consistent with (and relies on!) the 
recent paper of Melrose-Zhu \cite{MZ}, who obtain a similar type of expansion for the family of hyperbolic metrics 
on a degenerating hyperbolic surface; indeed, their result is one of the key ingredients here. 

The broader context of this paper is that the necessity of determining higher asymptotics of the \WPe metric
became apparent in the work that led up to \cite{JMMV}. The goal enunciated there is to study the natural
elliptic operators on $(\calM_\gamma, g_{\WP})$, for example, the Hodge Laplacian, twisted Dirac operators, etc.
Because of the singularities of $\overline{\calM}_\gamma$, the first step in any such study is to come to terms
with the effect on these operators of the singular structure of the metric along the divisors, and in particular to determine
whether these make it necessary to introduce new boundary conditions. It was shown in \cite{JMMV} that such
boundary conditions are unnecessary for the scalar Laplacian, i.e., the scalar Laplacian is essentially self-adjoint. 
In current work by the second author here and Gell-Redman, it is proved that the Hodge Laplacian on differential forms 
is also essentially self-adjoint; in other words, the natural action of these operators on $\calC^\infty_0(\calM_\gamma)$ 
has a unique self-adjoint extension in $L^2$. From these results one can go on to develop the spectral geometry and 
index theory for $(\calM_\gamma, g_{\WP})$, and indeed this is an area of ongoing investigation by the authors, Gell-Redman 
and others.  One particularly interesting goal is to prove a signature theorem relative to the \WPe metric; this 
would be a counterpart to the Gauss-Bonnet theorem of \cite{LSY}. 

Numerous people have been very helpful in teaching us about the Riemann moduli space, and in discussing 
various parts of the geometry below. We mention in particular Dan Freed, Lizhen Ji, Maryam Mirzakhani, Andras Vasy, 
Mike Wolf, Sumio Yamada, and in particular Scott Wolpert. We also thank Richard Melrose and Xuwen Zhu;
their paper \cite{MZ} appeared in the later stages of this research and clarified one part of our analysis
substantially. Our paper was 
initiated during a several month visit to the Stanford Math Department by the second author, funded by DFG through 
grant  Sw 161/1-1. R.M.\ was partially funded by the NSF Grant DMS-1105050. J.S.\ gratefully acknowledges the 
kind hospitality of the Department of Mathematics at Stanford University.

\section{Preliminaries on the Riemann moduli space and $g_{\WP}$}\label{sect:prelimin}
In this section we recall a number of well-known facts about the Riemann moduli space. We begin with the
more topological aspects, following the monograph by Farb and Margalit \cite{FM}.

Let $S$ be the {\it model surface}, i.e.\ an oriented, closed smooth surface of genus $\gamma\geq2$. 
The \emph{Teichm\"uller space} $\mathcal T_{\gamma}$ of surfaces of genus $\gamma$ is the set of equivalence classes  
$[(\Sigma,\phi)]$, where $\Sigma$ is a Riemann surface (the Riemannian metric, or conformal or complex structure, 
is suppressed from the notation), $\phi\colon S\to \Sigma$ is a diffeomorphism, called a \emph{marking}, and
\begin{equation*}
(\Sigma_1,\phi_1)\sim(\Sigma_2,\phi_2)
\end{equation*}
if there is an isometry $I\colon \Sigma_1 \to \Sigma_2 $ such that the maps $I$ and $\phi_2\circ\phi_1^{-1}$ 
are isotopic. The set $\mathcal T_{\gamma}$ is in bijective correspondence with the representation variety 
$\operatorname{DF}(\pi_1(S), \operatorname{PSL}(2,\R))/\operatorname{PSL}(2,\R)$, i.e., the space of conjugacy classes 
of discrete faithful representations of the fundamental group $\pi_1(S)$ into $\operatorname{PSL}(2,\R)$. The latter 
space carries a natural topology, induced by the  compact-open topology of $\operatorname{Hom}(\pi_1(S),
\operatorname{PSL}(2,\R))$, in terms of which $\mathcal T_{\gamma}$ is Hausdorff. 

There is a geometrically defined atlas of charts which make $\calT_\gamma$ a smooth manifold of dimension $6\gamma-6$.
To describe this, fix a maximal set of pairwise disjoint, oriented simple closed curves $\{c_1,\ldots,c_{3\gamma-3}\}$ on $S$. 
These decompose $S$ into a collection $\{P_1,\ldots,P_{2\gamma-2}\}$ of pairs of pants, i.e., spheres with three open disks 
removed. A hyperbolic metric on each $P_j$ is determined up to isometry by specifying an unordered triple of positive numbers,
corresponding to the lengths of the three boundary curves; these boundary curves are then geodesics for this hyperbolic metric.
We can then attach pairs of pants to one another along a common boundary component of the same length. There is 
a twist parameter $\omega$ when we attach any two boundary curves which takes values in $\RR$. Using all of this, one shows 
that an element of $\calT_\gamma$ is determined by the pair of $(3\gamma-3)$-tuples 
\begin{multline*}
(\ell_1,\ldots,\ell_{3\gamma-3})=(\ell_{\Sigma}(c_1)\ldots,\ell_{\Sigma}(c_{3\gamma-3}))\in\R_+^{3\gamma-3},  \\
\mbox{and}\quad (\omega_1,\ldots\omega_{3\gamma-3})\in\R^{3\gamma-3}
\end{multline*}  
By a result of Fricke \cite[Theorem 9.5]{FM}, the map which assigns to a point $[(\Sigma,\phi)]\in\mathcal T_{\gamma}$  its 
\emph{Fenchel-Nielsen coordinates} $(\ell_1,\ldots,\ell_{3\gamma-3},\omega_1,\ldots\omega_{3\gamma-3})$ is a homeomorphism to 
$\R_+^{3\gamma-3}\times\R^{3\gamma-3} \cong\R^{6\gamma-6}$.   Fenchel-Nielsen coordinates provide global coordinates 
for $\mathcal T_{\gamma}$ which depend on the chosen collection of curves. 

We are interested here in the \emph{Riemann moduli space}, which is the quotient of $\mathcal T_{\gamma}$ by the action of 
the \emph{mapping class group} $\operatorname{Map}(S)$ of the surface $S$. This is the infinite discrete group  of isotopy classes of orientation preserving diffeomorphisms of $S$. It is isomorphic to the group of 
outer automorphisms of $\pi_1(S)$. The action of $\operatorname{Map}(S)$ on $\mathcal T_{\gamma}$ is given by
\begin{equation*}
f\cdot[(\Sigma,\phi)]=[(\Sigma,\phi\circ\psi^{-1})],
\end{equation*}
where $\psi\colon S\to S$ is any diffeomorphism representing $f$. We define the Riemann moduli space $\mathcal M_{\gamma}$ 
as the quotient $\mathcal T_{\gamma}/\operatorname{Map}(S)$. This action is properly discontinuous \cite[Sect.~11.3]{FM},
hence $\mathcal M_{\gamma}$ is an orbifold. Away from orbifold singularities, Fenchel-Nielsen parameters provide a local 
coordinate system on $\mathcal M_{\gamma}$. 

The space $\mathcal M_{\gamma}$ is not compact. Indeed, letting any one length parameter $\ell_j$ tend to zero gives
a sequence of points in $\calM_\gamma$ which leaves every compact set.  The converse to this statement is provided by 
\medskip

\noindent {\bf Mumford's compactness criterion}  \emph{For any $\varepsilon>0$, define 
the \emph{$\varepsilon$-thick part} of $\Mg$
\begin{equation*}
\Mg^\varepsilon=\{[(\Sigma,\phi)]\in\mathcal M_{\gamma}\mid\min_{1\leq j\leq3\gamma-3}\ell_{\Sigma}(c_j)\geq\varepsilon\};
\end{equation*}
Then for every $\varepsilon>0$, $\Mg^\varepsilon$ is compact.}

\medskip

\noindent If $\varepsilon$ is sufficiently small, $\mathcal M_{\gamma}\setminus\Mg^\varepsilon$ 
is connected, so $\mathcal M_{\gamma}$ has precisely one end.

For the purposes of this paper, we recall an alternate approach to defining $\calT_\gamma$ and $\calM_\gamma$, following
\cite{Tromba}, which leads directly to the Weil-Petersson metric.  Regarding $\Si$ as a smooth surface, consider
the space $\calC^\infty(\Si; \operatorname{Sym}^2 T^* \Si)$ of symmetric $2$-tensors on $\Si$, and 
the open subset $\calC^\infty(\Si; \operatorname{Sym}^2_+ T^* \Si)$ of sections which are everywhere
positive definite.  A Riemannian metric $g$ on $\Si$ is an element of this latter space, and we write 
$\Met$ (or $\Met(\Si)$) the space of all such metrics.   The  group of orientation preserving diffeomorphisms $\Diff(\Si)$ acts on 
$\Met$ by pullback. The subgroup $\Diff_0(\Si)$, consisting of all  orientation preserving diffeomorphisms isotopic to the identity,
is the connected component of the identity in $\Diff(\Si)$, and $\Diff(\Si)/\Diff_0(\Si) = \mathrm{Map}\,(\Si)$. 
Next, to each Riemannian metric we associate its Gauss curvature function 
$K_g \in \calC^\infty(\Si)$.  The space $\Met^{-1}$ is the space of metrics with $K_g \equiv -1$
(this is nonempty because $\gamma \geq 2$).  There is a smooth action of $\Diff(\Si)$ on $\Met(\Si)$, 
and the alternate characterizations of the Teichm\"uller and Riemann moduli spaces are as the quotients
\[
\calT_\gamma = \Met^{-1}(\Si)/ \Diff_0(\Si), \quad \mbox{and}\quad \calM_\gamma = \Met^{-1}(\Si)/ \Diff(\Si).
\]
Points in either of these space are denoted as equivalence classes $[g]$.
In practice, one must actually introduce a finite-regularity topology on all of these objects; we have
used smooth metrics and diffeomorphisms for simplicity of statement since these quotients are independent
of which Banach topology we use.  We discuss this more carefully in \textsection \ref{sect:modulispaces}   below. 

The space $\Met(\Si)$ carries a natural $L^2$ (or Ebin) metric, defined as follows. Elements of
$T_g \Met$ are sections of $\calC^\infty(\Si; \operatorname{Sym}^2(\Si))$ (with no positivity conditions), and then,
for $h_1, h_2 \in T_g \Met$, 
\begin{equation}
\langle h_1, h_2 \rangle_{L^2} = \int_{\Si} \langle h_1, h_2 \rangle_g\, dA_g.
\label{Ebin}
\end{equation}
This metric restricts to $T_g \Met^{-1}$.   Now, if $g \in \Met^{-1}$, and if we consider 
the local diffeomorphism orbit through $g$, $\calO_g := \{ F^* g: F \in \Diff(\Si),\ F \approx \mbox{id}\}$, then
\[
T_g \calO_g = \{ L_X g:  X \in \calC^\infty(\Si; T\Si)\}.
\]
The orthogonal complement of this tangent space with respect to $\langle \, , \, \rangle_{L^2}$ is the space
\begin{equation}
\Stt(g) = \{ \kappa \in \calC^\infty(\Si; \operatorname{Sym}^2(\Si)): \delta_g \kappa = 0,\ \tr^g \kappa = 0\}.
\label{Sttdef}
\end{equation}
Its elements are the so-called transverse-traceless tensors; these are naturally identified with holomorphic
quadratic differentials on $(\Si,g)$ (i.e., holomorphic with respect to the complex structure compatible with $g$),
and $\dim \Stt(g) = 6\gamma -6$.    An open neighborhood of $0$ in $\Stt(g)$ provides a local chart
in (a finite orbifold cover) of $[g]$ in either $\calT_\gamma$ or $\calM_\gamma$, see \cite{Tromba}, \cite{MW}.
However, more simply, there is a canonical identification 
\[
T_{[g]} \Mg \cong T_{[g]} \calT_\gamma \cong \Stt(g)
\]
for any $\Met^{-1} \ni g \in [g]$. We may finally define the Weil-Petersson metric on either of these spaces by declaring 
that at any point $[g]$, 
\begin{equation}
\left. g_{\WP}\right|_{[g]} (\kappa_1, \kappa_2) = \langle \kappa_1, \kappa_2 \rangle_{L^2(\Si, g)},
\label{defWP}
\end{equation}
where the metric on the right is calculated on the Riemannian surface $(\Si,g)$. 

The Weil-Petersson metric is K\"ahler and everywhere negatively curved. Furthermore, the Riemann moduli space has finite 
diameter with respect to $g_{\WP}$, and it is a remarkable fact (closely related to Mumford's compactness criterion) that the 
corresponding metric completion is naturally identified with the algebro-geometric Deligne-Mumford compactification:
\[
\overline{\calM_\gamma}^{\WP} \cong \overline{\calM_\gamma}^{\mathrm{DM}}.
\]
This is the starting point for our work here.  We are interested in the precise behavior of $g_{\WP}$ on
approach to the points of $\overline{\calM_\gamma}$.  We now record the previously known results about 
$g_{\WP}$ in Fenchel-Nielsen coordinates near the divisors.  Any point $[g_0]$ in a $k$-fold intersection
$D_J = D_{i_1} \cap \ldots \cap D_{i_k}$ defines a complete hyperbolic metric on a noded surface; the nearby 
points $[g] \in  \calM_\gamma$ are hyperbolic surfaces on the compact surface $\Si$ where the corresponding
lengths $\ell(c_{j_i})$, $i = 1, \ldots, k$, are all small. We are assuming for simplicity that the $D_i$ are all distinct
here, i.e., $[g]$ does not lie on a double-point of one of the immersed divisors; furthermore, we also relabel indices so that 
$j_i = i$, $i \leq k$.  We now quote the result we intend to refine: in Fenchel-Nielsen coordinates
near $D_1 \cap \ldots \cap D_k$,  
\begin{equation}
g_{\WP} =\frac{\pi}{2} \sum_{j=1}^k \left( \frac{d\ell_j^2}{\ell_j} +\frac{ \ell_j^3}{4\pi^4} d\omega_j^2\right) + g_D + \eta,
\label{wpmetric1}
\end{equation}
where $g_D$ is the Weil-Petersson metric on the noded surface corresponding to that intersection
of divisors, and $\eta$ is a lower order error term.  The original version of this formula \cite{Masur}, by Masur, gave only
a quasi-isometric equivalence between the two sides; substantial sharpenings were obtained by Wolpert \cite{Wolpert}
and even further by Yamada \cite {Yamada}. The net result of these papers, particularly the last, is that
the remainder term $\eta$ contains no leading order terms, so that the $d\ell_j$ and $d\omega_j$ directions
are orthogonal at $D_j$ in an appropriately rescaled sense, and that  these length and twist directions
corresponding to two different curves $c_i$ and $c_j$ are orthogonal to leading order at $D_i \cap D_j$. 
Later work by Liu, Sun and Yau \cite{LSY} provided estimates of up to four derivatives of the metric coefficients;
their motivation was to compute the curvature of the so-called Ricci metric, $g_{\mathrm{Ric}} = - \mathrm{Ric}(g_{\WP})$.
Quite recently, Wolpert has proved in \cite{Wolpert2} a certain uniform differentiability of the metric coefficients; 
in the language introduced below, this is equivalent to the \emph{conormality} of the metric.

In much of this paper we shall work in local coordinates in $\Mgc$ near a point $[g] \in D_1 \cap \ldots \cap D_k$,
which we write as $w = (\ell_1,\omega_1, \ell_2, \omega_2, \ldots, \ell_k, \omega_k, y)$. Here each $(\ell_j,\omega_j)$
is the Fenchel-Nielsen length and twist pair and $y$ is any choice of smooth local coordinate on $D_1 \cap \ldots \cap D_k$.
Thus $w_{2j-1} = \ell_j$ and $w_{2j} = \omega_j$, $j = 1, \ldots, k$ and $w_i = y_i$, $i > 2k$. We are implicitly
using that $\Mgc$ has a natural smooth (in fact, analytic) structure near $[g]$.

We can now state our main theorem more precisely:
\begin{thm}\label{thm:mainthm}
Writing $g_{\WP}$ as in \eqref{wpmetric1}, then every coefficient $\eta_{pq}$ has a complete asymptotic expansion
\[\
\eta_{pq} \sim \sum_\alpha \sum_{j=0}^{N_\alpha}\ell^{\alpha/2} (\log \ell)^j \eta_{\alpha,j}(\omega,y). 
\]
Each $\alpha$ and $j$ is a multi-index of nonnegative integers, so $\ell^{\alpha/2} = \ell_1^{\alpha_1/2} \ldots \ell_k^{\alpha_k/2}$,
$(\log \ell)^j = (\log \ell_1)^{j_1} \ldots (\log \ell_k)^{j_k}$, and furthermore each $\eta_{\alpha,j} \in \calC^\infty$.
\end{thm}
\begin{rem}
As we describe carefully below, this result relies on the polyhomogeneity of a certain uniformizing conformal factor. This has
recently been proved by Melrose and Zhu \cite{MZ} on $\BbD^{\mathrm{reg}}$, i.e., away from the intersections of the
divisors. They anticipate that the extension of their result to all of $\BbD$ should be true and will follow by 
an elaboration of their same technique. Since that result has not appeared yet, our theorem above is only claimed at
present in that region of the compactified moduli space.  However, all other parts of the argument here are valid even 
near $\BbD^{\mathrm{sing}}$. 
\label{remarkMZ}
\end{rem}
Let us give a very brief sketch of the proof.  We first introduce the universal family $\calM_{\gamma,1}$, 
which is the space of Riemann surfaces with one marked point. This marked point traces out a copy of $\Si$ over each 
point $[g] \in \calM_\gamma$, and there is a forgetful map
\begin{equation}
\Pi: \calM_{\gamma,1} \longrightarrow \calM_\gamma;
\label{forgetful}
\end{equation}
the fibre $\Pi^{-1}([g])$ is a copy of that Riemann surface, i.e., of $\Si$ with an equivalence class of hyperbolic metrics. 
This is an orbifold fibration: up to finite covers of the domain and range, $\Pi$ is a fibration. Since these 
finite covers are irrelevant in our analysis, we often refer to \eqref{forgetful} as a fibration, with this understanding.
There is a corresponding map of Deligne-Mumford compactifications 
\begin{equation}
\overline{\Pi}: \overline{\calM_{\gamma,1}} \longrightarrow \overline{\calM_\gamma},
\label{forgetful2}
\end{equation}
which (again up to finite covers) is a {\it singular fibration}; in other words, it is a standard fibration away from
the divisors in the base, while if $[g_0] \in \overline{\calM_\gamma} \setminus \calM_\gamma$, then $\overline{\Pi}^{-1}([g_0])$
is a noded surface obtained by inserting a thrice-punctured (or thrice-marked) copy of $S^2$ between each of
the nodes of the degenerated surface $\Si_{[g_0]}$. 

The next step is to introduce a family of vertical hyperbolic metrics on $\Mgonec$, or in other words a choice of 
representative hyperbolic metric $g$ on the fiber $\Si_{[g]}$ for each $[g] \in \overline{\mathcal M_{\gamma}}$. 
Thus when $[g]$ lies in the interior, $g$ is an element in $\Met^{-1}$ 
representing $[g] \in \Met^{-1}/\Diff$, while if $[g] \in \Mgc\setminus \Mg$, then $g$ is a complete hyperbolic 
metric on the corresponding noded surface. There is probably not a completely natural choice of this representative, 
but the important feature here is that this choice depends smoothly on $[g]$ in the interior, and is polyhomogeneous
(say in the Fenchel-Nielsen coordinates) on the compactification.  In other words, in terms of one of the coordinate 
systems $w = (\ell,\omega,y)$ defined above, we seek a choice of representative $g$ depending smoothly on $w$ 
in the region where all $\ell_j > 0$ and which is polyhomogeneous as the $\ell_j \searrow 0$.  This has an obvious 
meaning away from collection of curves on $\Sigma$ which are degenerating at these divisors, and we refer to the next 
section for a full description. This local choice of metric representative is a slice of the diffeomorphism action, and 
we refer to any slice with these regularity properties as a polyhomogeneous slice.  

As we show later, it is in fact sufficient for most of the work below to work with a local polyhomogeneous slice where
the individual metrics are approximately hyperbolic, namely they are required to be hyperbolic on a neighborhood of 
the degeneration locus but have variable curvature elsewhere on $\Si$. These are easy to construct. 
Only at the final step do we appeal to the existence of a family of fibre-wise conformal factors $\varphi_{g}$ which
relate the approximately hyperbolic slice to an exactly hyperbolic one. The fact that these conformal factors depend
in a polyhomogeneous way on $[g]$ at $\BbD$ was proved recently by Melrose and Zhu \cite{MZ}. As already noted in
Remark \ref{remarkMZ}, we record the caveat that their paper proves this fact only away from the intersections of $\BbD$;
however, they expect their results to extend to the general case and have communicated that details should be forthcoming shortly. 

Having specified such a polyhomogeneous slice, we may then consider the bundle $\Stt$ over the coordinate 
neighborhood in $\Mgc$ whose fiber at $[g]$ consists of the transverse-traceless tensors $\Stt(g)$. 
Notice that the space $\Stt(g)$ only depends on the conformal class, i.e., is the same for $g$ and  $e^{2\varphi} g$. 
This is obviously a smooth bundle away from $\Mgc \setminus \Mg$, but we shall prove that there is a local frame of sections
$\kappa_i$ which are polyhomogeneous at the singular divisor in the base variables.  Since the area form $dA_g$ is 
also polyhomogeneous at this singular divisor, we see that every term in the integral on the right in \eqref{defWP}, cf.\ 
\eqref{Ebin}, is polyhomogeneous, when evaluated on this special frame.  Thus the matrix coefficients of $g_{\WP}$ in this
special frame are polyhomogeneous. The key point in finding a polyhomogeneous frame $\{\kappa_i\}$ is as follows. 
It is clear from the existence of a polyhomogeneous slice $g(w)$ that any partial derivative $\del_{w_i} g$ is polyhomogeneous, 
and each such infinitesimal variation of metrics is an element of $T_{g(w)} \Met^{-1}$. However, we must then project these to elements 
of $\Stt(g(w))$. The main issue then is to prove that the family of orthogonal projections
\[
T^g: T_g \Met^{-1} \longrightarrow T_{[g]} \Mg
\]
has a polyhomogeneous extension, in an appropriate sense, up to $\BbD$. 

We conclude by saying more about why we can work with only  an approximately hyperbolic slice until
the last step. We have already remarked that $\Stt$ is conformally invariant, and the inner product and area form transform
simply under conformal changes.  Thus if $g(w)$ is an approximately hyperbolic polyhomogeneous slice,
and if $\varphi(g(w))$ is the associated family of conformal factors so that $e^{2\varphi(g(w))}g(w)$ is hyperbolic, then we
can rewrite \eqref{Ebin} and \eqref{defWP} in terms of $g(w)$ as 
\begin{equation}
\langle \kappa_1, \kappa_2 \rangle_{\WP} = \int_\Sigma  \langle \kappa_1, \kappa_2 \rangle_{g(w)} e^{-2\varphi(w)} \, dA_{g(w)},
\label{defWP2}
\end{equation}
since the $\kappa_i$ are the same for either $g(w)$ or $e^{2\varphi(g(w))} g(w)$.  We then prove the main result as follows: 
each of the terms in this expression are polyhomogeneous, and the integral is a pushforward by a $b$-fibration, 
so by Melrose's pushforward theorem, the corresponding matrix coefficients of $g_{\WP}$ are polyhomogeneous in $w$. 
We explain this step later.

\section{Blowups and polyhomogeneous slices}\label{univbundle}
To describe and carry out the steps of the proof in more detail, we now introduce two auxiliary spaces which play a key 
role in this paper. These are resolutions of $\Mgc$, $\Mgonec$, and are the spaces on which the metrics, frames, etc., 
described at the end of \textsection \ref{sect:prelimin}  are actually polyhomogeneous. 

The first such space, $\Mgt$, is the {\it real} blowup of $\Mgc$ along the divisor $\DD$.  Assume, as before, 
that $[g]\in D_J$, where $J = \{1, \ldots, k\}$. Choose local holomorphic coordinates $(z_1, \ldots, z_N)$ 
in a neighborhood $\calU$ of $[g]$ so that each $D_j \cap \calU = \{z_j = 0\}$, $j = 1, \ldots, k$, and that 
$(z_{k+1}, \ldots, z_N)$ is a local holomorphic chart on $D_J$.  

Suppose first that $|J| = 1$, so $D_J = D_1$, and $[g]$ lies in a single divisor. The blowup of  $D_1 \cap \calU$
in $\calU$ is obtained by replacing each point $q \in D_1 \cap \calU$ by its normal circle bundle. This yields 
a manifold with boundary, where the boundary is the total space of a circle bundle over $D_1 \cap \calU$. More 
generally, when $|J| > 1$, this process is carried out along each one of the $D_j \cap \calU$, and it is clear from
the fact that the $D_j$ meet with simple normal crossings that the blowups in the different divisors are 
independent from one another. The part of the blowup over each intersection locus $D_J$ is a $(S^1)^{|J|}$ bundle.
This construction is well-defined globally and defines a manifold with corners which we call $\wtM_\gamma$.  To
be in accord with current usage, this space is actually not quite a manifold with corners for the simple reason
that some of the boundary hypersurfaces intersect themselves.  Since our considerations are local
on $\wtM_\gamma$, we may overlook this point, and for simplicity we may assume that we are working
on a manifold with corners. 

The interiors of the boundary hypersurfaces of $\wtM_\gamma$ are $S^1$ bundles over $\DD^{\reg}$, the set of points 
$q \in \DD$ which lie on only one divisor, and away from double-points. The codimension $k$ corners are $(S^1)^k$ 
bundles over the appropriate intersections of divisors $D_J$. Altogether, we write
\[
\Mgt = [ \Mgc;  \DD].
\]
This space has boundary hypersurfaces $H_j$, $j = 1,\ldots, N$, each corresponding to one of the divisors $D_j$, as
well as corners corresponding to the various intersections of the $D_j$.  We denote by $\beta$ the blowdown map $\Mgt \to \Mgc$. 

It is worth noting that the local holomorphic coordinates used here are not directly comparable to Fenchel-Nielsen coordinates, 
i.e., while we may set $y = (z_{k+1}, \ldots, z_N)$, it is not the case that each $(\ell_j,\omega_j)$ is a function of 
$\Re z_j, \Im z_j$. Nonetheless, the fibres $\beta^{-1}([g])$ for $[g] \in \DD$ are well-defined. In fact, defining the
polar coordinates $|z_j|$, $\arg z_j$, then both of the sets of coordinates $(|z|_1,\arg z_1, \ldots, |z|_k,\arg z_k, y)$
and $(\ell_1, \omega_1, \ldots, \ell_k, \omega_k, y)$ lift to smooth (actually, real analytic) coordinate systems on $\Mgt$, and 
these are  smoothly (real analytically) equivalent. 

We next consider the analogous construction for the compactified universal family $\Mgonec$.  As noted earlier,
up to finite covers, there is a singular fibration $\BPi: \Mgonec \to \Mgc$, and we wish to define a manifold with
corners obtained by blowing up certain submanifolds of $\Mgonec$ so that $\BPi$ lifts to a map
\[
\TPi:  \Mgonet \longrightarrow \Mgt
\]
which is a $b$-fibration between manifolds with corners.  To carry out this construction, we 
first blow down the noded $S^2$ components of the singular fibres $\overline\Pi^{-1}([g_0])$ in $\Mgonec$;
thus we replace these singular fibres by the union of noded Riemann surfaces which correspond to the 
complete hyperbolic metric $g_0$. The resulting space $(\Mgonec)'$ is the true universal family of $\Mgc$. (We could, 
of course, have bypassed the Deligne-Mumford compactification of $\Mgone$ and proceeded directly
to this smaller compactification.) The next step is to lift this singular fibration via the blowdown map
$\beta: (\Mgonec)' \to \Mgc$; the final step is to blow up the lifts of the nodes of these singular fibres in
the total space. The resulting space is denoted $\Mgonet$.

It is more transparent to describe all of this in local coordinates. For simplicity, consider this first near 
$\BbD^{\mathrm{reg}}$. Use the (lifted) Fenchel-Nielsen coordinates $(\ell,\omega,y)$ on $\Mgt$ as well as local coordinates 
$(\tau, \theta)$ on a neighborhood $\calC$ of the appropriate curve $c \subset \Si$ which degenerates at $[g_0]$.
We suppose that the local choice of approximately hyperbolic metrics $g$ on each fiber of $\BPi$ is adapted to these coordinates 
in the sense that $\calC$ is the open cylinder $\calC = (-1,1)_{\tau} \times S^1_{\theta}$ and 
\begin{equation}
\left. g \right|_{\calC} = \frac{d\tau^2}{\tau^2 + \ell^2} + (\tau^2 + \ell^2) \, d\theta^2.
\label{mhm}
\end{equation}
Note that \eqref{mhm} is indeed a hyperbolic metric and can be reduced to the more familiar form
\[
dt^2 + \ell^2 \cosh^2(t) \, d\theta^2
\]
by the change of coordinates $t = \operatorname{arcsinh} \tau/\ell$.  These cylindrical coordinates extend 
to the fibers of $(\Mgonec)'$ over $\DD$, but become dependent at the preimage of the nodal set
$\{ \ell = \tau = 0\}$.  This preimage has coordinates $y, \omega, \theta$, and its blowup reduces to the 
blowup of $(0,0)$ in $[0, \ell_0)_\ell \times (-1,1)_\tau$, with the remaining coordinates $(y, \omega, \theta)$ 
as parameters.  Let us use polar coordinates $\ell = \rho \sin \chi$, $\tau = \rho \cos \chi$, with
$\rho \geq 0$, $\chi \in [0,\pi]$. Altogether then, $(\rho, \chi, \omega, \theta, y)$ is
a coordinate system on $\Mgonet$. The new face $\rho = 0$ is called the front face.

It is useful to consider a reduced version of this construction where we replace the product $\calC \times [0,\ell_0)_\ell$ 
by its blowup
\begin{equation}
\wh{\calC} = [ \calC \times [0,\ell_0);  \{ \ell = \tau = 0\} ]. 
\label{defCh}
\end{equation}
(In other words, we suppress the coordinates $y$ and $\omega$ along the hypersurface boundary $H$ of $\Mgt$.
As before, the new face is called the front face of $\wh{\calC}$.  A neighborhood of the corresponding region in $\Mgonet$
is a product $\wh{\calC} \times S^1_\omega \times \calW_y$, where $\calW_y$ is a neighborhood of $[g_0]$ in $\BbD$. 
More generally, near crossing points of $\BbD$, this construction may be carried out independently near
each of the degenerating curves $c_j$. 


The space $\Mgonet$ has two types of boundary hypersurfaces. The first are closures of the hypersurfaces 
$H_j^{\reg} \times (\Sigma \setminus \fc_j)$, where the $H_j^{\reg}$ are the hypersurfaces in $\Mgt$ which
cover the components of $\BbD^{\reg}$, and the second are the faces $F_j$ obtained by blowing up 
$H_j^{\reg} \times \fc_j$.  Because of the self-intersections of the irreducible components of $\BbD$,
these boundary hypersurfaces of $\Mgonet$ may self-intersect at their boundaries, and for this reason,
$\Mgonet$ is slightly more general than a manifold with corners. However, our considerations are sufficiently
local that this does not affect anything here, and we shall think of this space as a manifold with corners.

The fibration $\Pi: M_{\gamma,1} \to M_\gamma$ extends to a $b$-fibration
\begin{equation}
\hat{\Pi}: \Mgonet \longrightarrow \wtM_\gamma.
\end{equation}
Briefly, a $b$-fibration is the most useful analogue of a submersion in the setting of maps between 
manifolds with corners. We refer to \cite{MelCCD} for the precise definition of $b$-fibrations.  
We appeal to this structure only once again, at the very end of this paper.  
The preimage of a point $q$ in the interior of $\Mg$ is the compact surface $\Sigma$. 
If $q \in H_j^{\reg}$, then $\wh{\Pi}^{-1}(q)$ is a union of the bordered surface obtained by
adding the boundary curves to $\Sigma \setminus \fc_j$ and the cylinder $[-1,1] \times S^1$,
which is the new face created by the final blowup.  We also define the vertical tangent bundle
$T^{\ver} \Mgonet$. Its fibres are the tangent planes to the fibres $\wh\Pi^{-1}(q)$ for $q$ in the interior;
over the boundary, however, these fibres are `broken', so the vertical tangent space is
either the tangent plane to the noded degeneration of $\Si$, or else to the tangent space of $\wh\calC$.

One of the key properties of this fibration is that the family of hyperbolic metrics $g_q$ on the vertical
tangent bundle extends naturally to a (degenerate) fibrewise metric on $\Mgonet$. Over the boundaries
of $\Mgonet$, these vertical hyperbolic metrics are either the complete finite area hyperbolic metrics,
over the noded degenerations of $\Si$, or else the complete (infinite area) hyperbolic metric
\[
\frac{dT^2}{1+T^2} + (1+T^2) d\theta^2
\]
on the front face of $\wh{\calC}$.   Melrose and Zhu \cite{MZ} prove the following 
\begin{prop}[\cite{MZ}] \label{prop:verticaltwotensor}
The family of metrics $g(w)$ on $\Sigma$ over the interior of $\Mg$ extends to a polyhomogeneous section of
the symmetric second power of the dual of the vertical tangent bundle of $\Mgonet$. 
\end{prop}
\begin{rem}
As we have already noted, \cite{MZ} only establishes this polyhomogeneity near points of $\BbD^{\mathrm{reg}}$, but
not at the intersection locus of the divisors. They expect to complete the proof in that case soon as well. 
\end{rem}

We conclude this section with a brief explanation for how to make the translation between the results in \cite{MZ}
and what is needed here. 
We begin with a polyhomogeneous family of approximately hyperbolic metrics $\hat{g}(w)$ and then find
the conformal factor $\varphi(w)$ such that $e^{2\varphi(w)}\hat{g}(w)$ is hyperbolic. The family $\hat{g}(w)$
has been constructed to be polyhomogeneous. We indicate now why the main theorem of \cite{MZ} shows
that $\varphi(w)$ is as well. 

The notation in \cite{MZ} is as follows. The annulus $A$ is identified with the quadric $\{(z,w): zw = t\}$;
for simplicity here we assume that $t \in \RR$, $0 < t < 1/4$; this parameter is equivalent to the length
parameter $\ell$, see below. We parametrize half of this region by the annulus
$\{\sqrt{t}  \leq |z| \leq 1/2\}$. The family of hyperbolic metrics here is given as
\[
\left( \frac{\pi \log |z|}{\log t}\csc  \frac{\pi \log |z|}{\log t} \right)^2 \frac{|dz|^2}{|z|^2 (\log |z|)^2}.
\]
This agrees with \eqref{mhm} upon making the substitutions
\[
\ell = \frac{\pi}{|\log t|}, \qquad \frac{\tau}{\ell} = \cot \frac{\ell}{|\log |z||}.
\]
In \cite{MZ}, the radial variable $|z|$ is replaced by $1/|\log |z||$, so this change of variables shows
that polyhomogeneity (and indeed log smoothness) in this new logarithmic variable, as
proved in \cite{MZ}, is equivalent to polyhomogeneity (log smoothness) in $\tau$. 

\section{Global analysis and $\calM_\gamma$}\label{sect:modulispaces}
We now recall some standard facts about deformations of hyperbolic metrics on surfaces. The point of view adopted
here is the one promoted by Tromba \cite{Tromba}, and is the specialization to this low dimension of the deformation
theory of Einstein metrics. 

\subsection{Curvature equations and Bianchi gauge}
Consider the operator which assigns to a metric its Gauss curvature: $g\mapsto K^g$. This is a second order nonlinear 
differential operator, and since $\operatorname{Ric}^g=K^g g$ in dimension $2$, metrics of constant curvature are the same as Einstein
metrics in this setting. 

Let $(\Sigma^2,g_0)$ be a closed surface where $K^{g_0} \equiv -1$.  Nearby hyperbolic metrics correspond to solutions of 
\begin{equation}\label{eq:Einsteineq}
S^2(\Sigma, T^*\Sigma) \ni h \mapsto E^{g_0}(h)\coloneqq(K^{g_0+h}+1)\cdot(g_0+h)=0
\end{equation} 
with $h$ suitably small. The nonlinear operator $E^{g_0}$ is called the \emph{Einstein operator}. Since $\Sigma$ is compact, we can 
let $E^{g_0}$ act between appropriate Sobolev or H\"older spaces, but for simplicity we do not specify the function spaces precisely 
until necessary. 

The operator $E^{g_0}$ is not elliptic because it is invariant under the infinite-dimensional group $\operatorname{Diff}(\Sigma)$ of diffeomorphisms of $\Sigma$.
For any metric $g$, the tangent space of the $\operatorname{Diff}(\Sigma)$ orbit through $g$ consists of all symmetric $2$-tensors of the form 
$\mathcal L_X g$, where $X$ is a vector field on $\Sigma$, or equivalently, as $(\delta^{g})^* \omega$, where $\omega$ is the $1$-form
metrically dual to $X$. Here $\delta^g\colon S^2(\Sigma, T^*\Sigma) \to \Omega^1(\Sigma)$ is the divergence operator and 
$(\delta^{g_0})^*$ is its adjoint; in local coordinates, 
\[
((\delta^g)^{\ast}\omega)_{ij} =\frac{1}{2}(\omega_{i;j}+\omega_{j;i}).
\]
Note that $\tr^g(\delta^g)^{\ast}=-\delta^g\colon \Omega^1\to\Omega^0$, where $\delta^g\colon \Omega^1(\Sigma) \to \Omega^0(\Sigma)$ is the 
standard codifferential. The \emph{conformal Killing operator} is the projection of $(\delta^g)^{\ast}$ onto its trace-free part: 
\[
\mathcal D^g\omega\coloneqq (\delta^g)^{\ast}\omega+\frac{1}{2}\delta^g(\omega)g\colon\Omega^1(\Sigma)\to S_0^2(\Sigma, T^*\Sigma).
\]
This is the adjoint of $\delta^g\colon S_0^2(\Sigma, T^*\Sigma)\to\Omega^1(\Sigma)$.  It follows from all this that the nullspace of $\delta^g$ 
on $S^2(\Sigma, T^*\Sigma)$ equals the $L^2$-orthogonal complement of the tangent space of the diffeomorphism orbit passing through 
$g$, and furthermore that the system 
\[
h\mapsto(E^{g}(h),\delta^{g}(h))
\]
is elliptic. It is convenient to consider instead the single operator 
\[
N^g(h)\coloneqq E^g(h)+(\delta^{g+h})^{\ast}B^g(h)=(K^{g+h}+1)(g+h)+(\delta^{g+h})^{\ast}B^g(h),
\]
where 
\begin{equation}\label{eq:Bianchioperator}
h\mapsto B^{g}(h)\coloneqq\delta^{g}(h)+\frac{1}{2}d\tr^{g}h
\end{equation}
is the \emph{Bianchi operator}.  We say that $h$ is in Bianchi gauge if $B^{g}(h)=0$. Clearly, if $E^g(h) = 0$ and $B^g(h) = 0$, then 
$N^g(h) = 0$.  The converse, which is due to Biquard, is true as well.
\begin{prop}
Suppose that $g_0$ is hyperbolic. If $h\in S^2(\Sigma, T^*\Sigma)$ is sufficiently small and $N^{g_0}(h)=0$, then $g_0+h$ has constant 
Gauss curvature $-1$ and $B^{g_0}(h)=0$.
\end{prop}
Before recalling Biquard's proof, recall that if $h = fg$ is pure trace, then 
\begin{equation}\label{eq:Bianchitracefree}
B^{g}(fg)=\delta^{g}(fg)+\frac{1}{2}d\tr^{g}(fg)=-df+df=0. 
\end{equation}
Thus applying $B^{g+h}$ to $N^{g}(h)$ yields 
\begin{equation}
B^{g+h}(E^g(h)) = B^{g+h} (\delta^{g+h})^* B^g(h) = 0.
\label{bt}
\end{equation}  
For any metric, write 
\begin{equation}
P^g\coloneqq B^g \circ (\delta^g)^{\ast}\colon\Omega^1(\Sigma)\to\Omega^1(\Sigma);
\label{defP}
\end{equation}
by a standard Weitzenb\"ock identity, 
\begin{equation}\label{eq:Weitzenboeck}
P^g=\frac{1}{2}(\Delta^g-2K^g),
\end{equation}
where $\Delta^g$ is the Hodge Laplacian on $1$-forms. 
\begin{proof}
It suffices to establish that, if $g = g_0$ is hyperbolic, then $\omega=B^{g_0}(h) =0$. By \eqref{bt}, $P^{g_0+h}\omega = 0$, or equivalently, by
\eqref{eq:Weitzenboeck}, 
\begin{equation*}
\langle 2P^{g_0+h}\omega,\omega\rangle=\|\nabla^{g_0+h}\omega\|^2-2K^{g_0+h}\|\omega\|^2 = 0. 
\end{equation*}
However, since $K^{g_0}=-1$, then if $h$ is sufficiently small, $K^{g_0+h}<0$. We conclude that $\omega=0$, 
and thus $E^{g_0}(h)=0$, as desired. 
\end{proof}

\subsection{Linearized curvature operators}
The linearizations of the curvature operators above are not hard to compute. As before, assume throughout 
that $K^{g_0} \equiv -1$. 
If $h=h^0+fg_0$ is the decomposition into trace-free and pure-trace parts, then 
\begin{equation}\label{eq:variationGausscurv}
DK^{g_0}(h)=(\frac{1}{2}\Delta^{g_0}+1)f+\frac{1}{2}\delta^{g_0}\delta^{g_0}h^0.
\end{equation}

We see directly from this that
\[
DE^{g_0}(k)=\big((\frac{1}{2}\Delta^{g_0}+1)f+\frac{1}{2}\delta^{g_0}\delta^{g_0} h^0\big)g_0, 
\]
and furthermore, 
\begin{equation}\label{eq:operatorLg}
DN^{g_0}(h) \coloneqq L^{g_0}(h)=\frac{1}{2}(\nabla^{\ast}\nabla-2)h^0+\big(\frac{1}{2}(\Delta^g+2)f\big)g.
\end{equation}
We call $L^{g_0}$ the \emph{linearized Bianchi-gauged Einstein operator}. Note finally that 
by differentiating the identity $B^{g+h}N^g(h)=P^{g+h}B^g(h)$ at $h=0$, we obtain
\[
B^gL^g=P^gB^g. 
\]
 
\subsection{Transverse-traceless tensors}
A key object in this paper is the space $\Stt =\Stt(g_0)$ of transverse-traceless tensors on the surface $\Sigma$ with
respect to the metric $g_0$; by definition, this is the nullspace of $L^{g_0}$.  This space represents the tangent space of $\calM_\gamma$ at $g_0$ and depends only on the conformal class $[g_0]$, cf.~Proposition \ref{prop:conformalinv} below. 

Using \eqref{eq:operatorLg}, $L^{g_0}(h^0+fg_0) = 0$ if and only if $f=0$ and $(\nabla^{\ast}\nabla-2)h^0=0$; by \eqref{eq:Weitzenboeck},
the second condition is equivalent to $\delta^{g_0} h^0=0$. Therefore 
\[
\ker L^{g_0} = \Stt(g_0) =\{h\in S^2(\Sigma, T^*\Sigma)\mid \delta^{g_0} h=0, \tr^{g_0} h=0\}
\]
and hence $\Stt$ is the tangent space to the submanifold 
of metrics with constant curvature $K \equiv -1$ in Bianchi gauge with respect to $g_0$. In particular, $\Stt$ is orthogonal 
to the diffeomorphism orbit through $g_0$.  It is straightforward to check that since $\dim \Sigma = 2$, $\delta^{g_0}$ is elliptic 
as a map between sections of $S^2_0(\Sigma, T^*\Sigma)$ and $\Omega^1(\Sigma)$, hence $\dim \Stt< \infty$ and consists of smooth elements. 
In fact
\begin{equation*}
\dim\Stt^g=6(\gamma-1);
\end{equation*}
this holds because there is a canonical identification of $\Stt$ with the space of holomorphic quadratic differentials on $\Sigma$. 

\begin{prop}\label{prop:conformalinv}
When $\dim \Sigma = 2$, $\Stt$ is conformally invariant; in other words, if $g_1=e^{2u}g$ then 
\begin{equation*}
\tr^{g_1}h=0,\delta^{g_1}h=0\quad\Longleftrightarrow\quad\tr^gh=0, \ \delta^gh=0.
\end{equation*}
\end{prop}
\begin{proof}
The first part of the condition is obvious since $\tr^{g_1}h=e^{-2u}\tr^gh$.  The rest follows from the general identity   
\begin{equation*}
\delta^{g_1}h=e^{-2u}(\delta^gh+(\tr^gh)du+(2-n)\iota(\nabla^gu)h),
\end{equation*}
so since $n=2$,  if $\tr^gh=0$ and $\delta^gh=0$, then $\delta^{g_1}h=0$, as claimed.
\end{proof}

\subsection{Local deformation theory}
It is now standard to deduce some features of the local deformation theory of hyperbolic metrics. We refer to 
  \cite{MW} and \cite{Tromba} for complete proofs. 

If $(\Sigma,g_0)$ is a closed hyperbolic surface, as above, we describe the Banach space structure of all nearby Riemannian metrics $g$ 
(in the $\calC^{2,\alpha}$ topology) with $K_g  = -1$, and identify those metrics in this Banach submanifold which are in
a local slice with respect to the action of $\operatorname{Diff}(\Sigma)$. 

Write $\mathcal M^{\operatorname{hyp}}$ for the set of $\mathcal C^{2,\alpha}$ Riemannian metrics with curvature $-1$. 
\begin{lemma}
The space $\calM^{\operatorname{hyp}}$ is a Banach submanifold in $\calC^{2,\alpha}S^2(\Sigma ,T^{\ast}\Sigma)$. 
\end{lemma}
By the discussion above, 
\begin{multline*}
T_{g_0}\calM^{\operatorname{hyp}}=\{h\in \calC^{2,\alpha}S^2(\Sigma ,T^{\ast}\Sigma)\mid  DK^{g_0}h=0\}\\
=\{h=h^0+fg_0\in \calC^{2,\alpha}S^2(\Sigma,T^{\ast}\Sigma) \mid \tr^{g_0}h^0=0,\ (\Delta^{g_0}+2)f+\delta^{g_0}h^0=0\}.
\end{multline*} 
We now impose the gauge condition. 
\begin{lemma}
There is a constant $\varepsilon$ depending on $g_0$ such that the intersection of
\begin{equation*}
\calS_{g_0,\varepsilon} \coloneqq \calM^{\operatorname{hyp}} \cap\{g_0+h: B^{g_0}h=0,\ \|h\|_{2,\alpha}<\varepsilon\}
\end{equation*}
with the orbit of $\operatorname{Diff}^{3,\alpha}(\Sigma)$ is transverse at $g_0$. The space $\calS_{g_0,\varepsilon}$ is identified with the space of 
solutions $h$ to $N^{g_0}(h)=0$ with $\|h\|_{2,\alpha}<\varepsilon$, and its tangent space at $g_0$ equals $\Stt(g_0)$. 
\end{lemma}

There is also a local slice theorem.  
\begin{thm}\label{thm:locslice}
There is a positive constant $\varepsilon$ depending on $g_0$ and a neighborhood $\calU$ of $\mathrm{id}$ in 
$\operatorname{Diff}^{3,\alpha}(\Sigma)$ such that the map
\begin{equation*}
\calU\times \calS_{g_0,\varepsilon}\ni (F,h)\mapsto F^{\ast}(g_0+h)
\end{equation*}
is a diffeomorphism onto some neighborhood of $g_0$ in $\calM^{\operatorname{hyp}}$. 
\end{thm}  
In fact, since the genus $\gamma > 1$, one can show that the action of the entire connected component of the identity
of the diffeomorphism group acts properly. 
 
\subsection{The Weil-Petersson metric} 
We now use this formalism to describe the Weil-Petersson metric. 

Let $t\mapsto g_t$, $|t| <\varepsilon$, be a smooth path of hyperbolic Riemannian metrics through $g_0$. For $\varepsilon$ small, 
then Theorem \ref{thm:locslice} shows that there exists a unique $h_t\in S^2(\Sigma,T^{\ast}\Sigma)$ such that $N^{g_0}(h_t)=0$ and $h_t=F_t^{\ast}g_t$ for some diffeomorphism $F_t$, so in particular
the differential $\kappa:= \frac{d\, }{dt}\big|_{t=0}h_t$ lies in $\Stt(g_0)$. We now compute $\kappa$ in terms of 
$\dot g=\frac{d\, }{dt}\big|_{t=0}g_t$. Indeed, there exists a smooth family of diffeomorphisms $F_t$ and 
a smooth family of scalar functions $u_t$ such that 
\begin{equation}\label{eq:locslice}
g_0 +h_t=F_t^{\ast}e^{2u_t}g_t.
\end{equation}
Write $X=\frac{d\, }{dt}\big|_{t=0}F_t$ and $\dot u=\frac{d\, }{dt}\big|_{t=0}u_t$. Differentiating \eqref{eq:locslice} with 
respect to $t$ at $t=0$ yields
\begin{equation}
\kappa = L_Xg_0+2\dot ug_0+\dot g=(\delta^{g_0})^{\ast}\omega+2\dot u\, g_0+\dot g,
\label{kappa}
\end{equation}
where $X^\flat = \omega\in\Omega^1(\Sigma)$ is the $1$-form dual to $X$ with respect to $g_0$.  

We first determine $\omega$ as follows. Apply $B^{g_0}$ to \eqref{kappa} and recall \eqref{defP}, to get
\[
P^{g_0} \omega = - B^{g_0} \dot g,
\]
since $B^{g_0} \kappa = 0$. By \eqref{eq:Weitzenboeck}, $P^{g_0}$ is invertible, so denoting its inverse by $G^{g_0}$, we 
can write 
\begin{equation*}
\omega=-G^{g_0} B^{g_0} \dot g. 
\end{equation*}
Finally, with $\pi^{g_0}\colon S^2(\Sigma,T^{\ast}\Sigma)\to S_0^2(\Sigma,T^{\ast}\Sigma)$ the orthogonal projection onto trace-free tensors, we obtain that
\[
\kappa = - \pi^{g_0} (\delta^{g_0})^{\ast}  G^{g_0} B^{g_0} \dot g + \pi^{g_0} \dot g.
\]
In other words, $\kappa = T^{g_0} \dot g$, where the $L^2$-orthogonal projection $L^2 S_0^2(\Sigma,T^*\Sigma) \to\Stt$
is given by 
\begin{equation}
T^{g_0}=\pi^{g_0} \circ \left( \id-(\delta^{g_0})^{\ast}G^{g_0} B^{g_0}\right).
\label{defT}
\end{equation}

This now gives the formul\ae\ for the Weil-Petersson norm and inner product: 
\begin{equation}\label{eq:WPmetric}
\| \kappa\|^2_{\WP} = \| T^{g_0} \dot g \|^2_{L^2(\Sigma,dA_{g_0})}, \quad
\langle \kappa_1, \kappa_2 \rangle_{\WP} = \langle T^{g_0} \dot g_1, T^{g_0} \dot g_2 \rangle_{L^2(\Sigma,dA_{g_0})}.
\end{equation}



Our goal in the remainder of this paper is to analyze these expressions near the singular divisors of $\calM_\gamma$. 
Our main result is that there is a basis of sections of $\Stt$ which depends in a polyhomogeneous manner on $g_0$.
The main issue is to prove that the operator $G^{g_0}$ is polyhomogeneous in $g_0$.

\section{Hyperbolic cylinders}
The first step in this analysis is to study this problem for a model family of finite hyperbolic cylinders $(\calC,g_\ell)$ 
where the length $\ell$ of the central geodesic decreases to $0$. We have already described these metrics in \eqref{mhm} above,
and we note here that either component of the boundary $\del \calC = \{\tau= \pm 1\}$ is at distance 
$\operatorname{arcsinh} (1/\ell) \sim|\log(\frac{\ell}{2})|$ from the central geodesic $\{\tau=0\}$.

We consider here the family of inverses $G^{g_{\ell}}$ to the operators $P^{g_{\ell}}$ on $\calC$ and shall prove that it is polyhomogeneous 
in a particular sense as $\ell \searrow 0$. This analysis is explicit and we use it later to construct a parametrix for $P^g$ 
on families of degenerating compact hyperbolic surfaces. 

It is here that the blowup $\wh{\calC}$ of $\calC \times [0,\ell_0)$, which we introduced in \eqref{defCh}, becomes 
important. Indeed, the family $G^{g_\ell}$ has a rather complicated structure at $\ell = 0$; the most precise description of
this behaviour, which was studied carefully in \cite{ARS}, is that it is polyhomogeneous on a space obtained by a sequence of blowups 
of $\calC \times \calC \times [0, \ell_0)$. Fortunately we do not need all of this, and shall only require a much weaker form
of this, which we can prove directly. Namely, we show that if the family of $1$-forms $h_\ell$ on $\calC \times [0,\ell_0)$ vanishes 
in a neighbourhood $|\tau| \leq c$ for all $\ell$ and is polyhomogeneous on this product space at $\ell = 0$, then the family of 
solutions $\omega_\ell = G^{g_\ell} h_\ell$ lifts to be polyhomogeneous on $\wh{\calC}$.  We say in either case that $h_\ell$ and 
$\omega_\ell$ are polyhomogeneous in $\ell$, the passage to $\wh{\calC}$ being implied.  We can sidestep the full machinery 
of \cite{ARS} precisely because $h_\ell$ vanishes near $\tau = 0$, since this means that $G^{g_\ell} h_\ell$ does not depend on the 
behaviour of the Schwartz kernel $G^{g_\ell}(\tau, \tilde{\tau}, \ell)$ near $\tau = \tilde{\tau} = \ell = 0$, where its structure 
is particularly complicated. 

\subsection{The operator $P^{g_{\ell}}$}\label{subsect:operatorP}
We now compute the action of $P^{g_{\ell}}$ on $1$-forms.  This is simplest if we express any such form as a
combination $u \rho_1 + v \rho_2$, where $\rho_1 = \sigma_1 - i \sigma_2$, $\rho_2 = \sigma_1 + i \sigma_2$, and 
\begin{equation}
\sigma_1=\frac{d\tau}{\sqrt{\tau^2+\ell^2}},\qquad\sigma_2=\sqrt{\tau^2+\ell^2}\,d\theta.
\end{equation}
Introducing the Fourier decompositions $u = \sum u_k e^{ik\theta}$, $v = \sum v_k e^{ik\theta}$, then a short
computation gives that the operator induced by $P^{g_{\ell}}$ on the $k^{\mathrm{th}}$ Fourier mode of the column 
vector $(u_k, v_k)^T$ is
\begin{equation}\label{eq:operatorQ}
P_{\ell,k} =\frac{1}{2}\begin{pmatrix} P_{\ell,k}^+&0\\0& P_{\ell,k}^-\end{pmatrix},
\end{equation}
where 
\[
P_{\ell,k}^\pm = -(\tau^2 + \ell^2)\frac{d^2\,}{d\tau^2} -2\tau \frac{d\,}{d\tau} + 1 + \frac{(\tau \pm k)^2}{\tau^2+\ell^2}. 
\]

We are only interested in solving $P^{g_{\ell}} \, \omega = h$, and understanding the precise asymptotics
of this solution as $\ell \searrow 0$, when $h$ is polyhomogeneous and also vanishes near $\tau = 0$, 
and this provides a substantial simplification of the discussion below.  This analysis separates into one 
for the eigenmode $k=0$ and another for the sum of all the other eigenmodes together. 

The latter case is slightly simpler.
\begin{prop}\label{prop:inversenonzeromodes}
Suppose that $h$ is polyhomogeneous on $\wh\calC$ and vanishes for $|\tau| \leq c < 1$. Suppose too 
that the eigenmode $h_0 = 0$. Then the unique solution to $P^{g_{\ell}} \omega = h$ with $\omega(\pm 1, \theta) = \eta_\pm$
specified (and polyhomogenous in $\ell$) is polyhomogeneous on $\wh\calC$ and vanishes rapidly at the front
face of $\wh{\calC}$. 
\end{prop}

\begin{proof}
We have reduced the problem to a scalar one, involving the operators $P_{\ell,k}^\pm$, and so we first prove 
the result for any one of these. Let $P_{\ell,k}^+ \omega_k = h_k$, where $\omega_k(\pm1)=\eta_{\pm_k}$ and  
$h_k$ vanishes for $|\tau|\leq c<1$. Choose a constant $C>0$, independent of $k$, such that
\begin{equation*}
C\geq\max\{\|h_k\|_{L^{\infty}},\omega(\pm1)\}.
\end{equation*}
It follows that the 
function $\tilde\omega_k=\omega_k-C$ satisfies $\tilde\omega_k(\pm1)\leq0$ and
\begin{equation*}
P_{\ell,k}^+\tilde\omega_k=h_k-C\left(1+\frac{(\tau+k)^2}{\tau^2+\ell^2}\right)\leq h_k-C\leq0.
\end{equation*}
Thus by the maximum principle, $\tilde\omega_k\leq0$, or equivalently $\omega_k\leq C$.  Similarly,
$\omega_k \geq -C$.  Now we find a sharper barrier function. Set
\begin{equation*}
\zeta_k=Ce^{\alpha|k|\big(\frac{1}{c}-\frac{1}{|\tau|}\big)},
\end{equation*}
where $C$ is the same constant as above, and compute that
\begin{multline*}
P_{\ell,k}^+ \zeta_k = \left( \Big( \frac{1}{\tau^2+\ell^2}-\frac{\alpha^2(\tau^2+\ell^2)}{\tau^4}\Big)k^2  \right. + \\
\left. \Big (\frac{2\alpha\ell^2}{\tau^3}+\frac{2\tau}{\tau^2+\ell^2} \Big)|k|\tau +1+\frac{\tau^2}{\tau^2+\ell^2}\right) 
\zeta_k \geq 0
\end{multline*}
for all $k \neq 0$ and $|\tau| \leq  1$ provided $\alpha \in (0,1)$ is chosen sufficiently small.  
Consider the homogeneous equation $P_{\ell,k}^+ \omega_k=0$ on $|\tau|\leq c$. Then 
\begin{equation*}
\zeta_k(\pm c)=C\geq |\omega_k| \quad\textrm{and}\quad P_{\ell,k}^+(\zeta_k-\omega_k)\geq 0, \quad P_{\ell,k}^+(\omega_k - \zeta_k)\leq 0, 
\end{equation*}
hence
\begin{equation*}
|\omega_k| \leq \zeta_k.
\end{equation*}
This same estimate holds also for $P_{\ell,k}^-$. 

Altogether, summing over all nonzero $k\in\Z$, and using that $\omega_0 = 0$, we see that the solution to the original problem 
$P^{g_{\ell}}\omega=h$ satisfies 
\begin{equation*}
|\omega(\tau,\theta)|\leq\sum_{k\neq0}|\omega_k(\tau)|
\leq\sum_{k\neq0} Ce^{\alpha|k|\big(\frac{1}{c}-\frac{1}{|\tau|}\big)}\leq\frac{Ce^{\alpha\big(\frac{1}{c}-\frac{1}{|\tau|}\big)}}{1-e^{\alpha\big(\frac{1}{c}-\frac{1}{|\tau|}\big)}},
\end{equation*}
and this decays rapidly as $\tau\to 0$, uniformly in $\ell$. 
\end{proof}

We now discuss the case $k=0$ and analyze the family of operators
\begin{equation*}
P_{\ell,0}^+=P_{\ell,0}^-=-(\tau^2+\ell^2)\frac{d^2}{d\tau^2}-2\tau\frac{d}{d\tau}+1+\frac{\tau^2}{\tau^2+\ell^2}.
\end{equation*}
The change of variables  $T=\frac{\tau}{\ell}$ transforms this to 
\begin{equation*}
P_0=-(T^2+1)\frac{d^2}{dT^2}-2T\frac{d}{dT}+1+\frac{T^2}{T^2+1}.
\end{equation*}
This has two linearly independent homogeneous solutions 
\begin{equation*}
u_{\ast}=\sqrt{T^2+1},\qquad v_{\ast}=\frac{1}{\sqrt{T^2+1}}\left(T+\arctan(T)+T^2\arctan(T)\right),
\end{equation*}
hence $u_{\ell}(\tau)=u_{\ast}(\frac{\tau}{\ell})$ and $v_{\ell}(\tau)=v_{\ast}(\frac{\tau}{\ell})$ are solutions of $P_{\ell,0}^{\pm}u=0$. 
By direct inspection, these functions are polyhomogeneous on $\wh{\calC}$. 

We now write the unique solution to the inhomogeneous problem $P_{\ell,0}^{\pm}\omega =h$ with boundary values 
$\omega(\pm1)=\eta^{\pm}$ as
\begin{equation}\label{eq:explicitinverse}
\begin{split}
\omega(\tau) & = A u_\ell + B v_\ell \  +  \\ 
&\left( \frac{u_\ell(\tau)}{2\ell}\int_{-1}^{\tau} v_{\ell}(\sigma)\frac{h(\sigma)}{\sigma^2 + \ell^2}\,
d\sigma - \frac{v_\ell(\tau)}{2\ell}\int_{-1}^{\tau} u_{\ell}(\sigma)\frac{h(\sigma)}{\sigma^2 + \ell^2}\,d\sigma \right).
\end{split}
\end{equation}
The constants $A=A(\ell,h,\eta^{\pm})$ and $B=B(\ell,h,\eta^{\pm})$ are given by 
\begin{equation*}
\begin{pmatrix}A \\B \end{pmatrix}=\frac{1}{D}\begin{pmatrix} v_{\ell}(-1)&-v_{\ell}(1)\\-u_{\ell}(-1)&u_{\ell}(1)  \end{pmatrix} 
\begin{pmatrix}\eta^++u_{\ell}(1)I_1-v_{\ell}(1)I_2\\\eta^-+u_{\ell}(-1)I_1-v_{\ell}(-1),
I_2\end{pmatrix}.
\end{equation*}
where 
\begin{equation*}
I_1=\frac{1}{2\ell}\int_{-1}^1v_{\ell}(\tau)h(\tau)\,d\tau,\qquad I_2=\frac{1}{2\ell}\int_{-1}^1u_{\ell}(\tau)h(\tau)\,d\tau, 
\end{equation*}
and $D=u_{\ell}(1)v_{\ell}(-1)-u_{\ell}(-1)v_{\ell}(1)$. 

\begin{prop}\label{eq:inversecylinder}
Suppose that  $h$ is polyhomogeneous on $\wh{\calC}$ and vanishes for $|\tau|\leq c<1$. Then the unique solution to 
$P_{\ell,0}^{\pm}\omega=h $ with given boundary values $\omega (\pm1)=\eta^{\pm}$ (which may also be 
polyhomogeneous in $\ell$) lifts to a polyhomogeneous function on $\wh{\calC}$. 
\end{prop}
\begin{proof}
We have already noted that $u_\ell$ and $v_\ell$ are polyhomogeneous, so we may concentrate on the other terms in the
solution.  By assumption, $h$ is polyhomogeneous. Since $h$ vanishes for $|\tau| \leq C$,
it is also clear that each of the integrals, including the ones in the definitions of $I_1$, $I_2$, also has an expansion in  
$\ell$ as $\ell \searrow 0$. This proves the claim.
\end{proof}

As a final remark, let us compute 
\[
\pi_{\ell}\circ(\delta^{g_{\ell}})^{\ast} (\omega(\tau)\sigma_1) 
\]
in the region $|\tau|\leq c$, for $\omega = u_\ell$ or $v_\ell$; observe that this is the expression in the projection 
formula \eqref{defT}. We see that
\begin{equation*}
\pi_{\ell}\circ(\delta^{g_{\ell}})^{\ast}u_{\ell}\sigma_1=0\ \ \textrm{and}\quad \pi_{\ell}\circ(\delta^{g_{\ell}})^{\ast}v_{\ell}\sigma_1=
-\frac{2\ell^2}{(\tau^2+\ell^2)^2}\,d\tau^2+2\ell^2\, d\theta^2; 
\end{equation*}
note that these are scalar multiples of the transverse-traceless tensor $\kappa_{\ell,0}$ which appears in \eqref{eq:tttensors} below.
Similarly, with $\omega = u_\ell, v_\ell$, then $\omega(\tau)\sigma_2$ is mapped to a multiple of the transverse-traceless tensor 
$\nu_{\ell,0}$.

\subsection{Symmetric transverse-traceless two-tensors}\label{sect:symmtttensors}
We explicitly determine the space $\mathcal S_{\operatorname{tt}}^{g_{\ell}}$ of symmetric transverse-traceless $2$-tensors for the 
hyperbolic metric $g_{\ell}$ on $\calC$ and study its limit as $\ell\searrow0$.\\ 
\noindent\\
We still work in the  orthonormal  frame of one-forms  $\{\sigma_1,\sigma_2\}$ from \textsection \ref{subsect:operatorP} and let 
$\{\sigma_1^2-\sigma_2^2,\sigma_1\otimes\sigma_2+\sigma_2\otimes\sigma_1\}$ be the induced frame of the bundle 
$S_0^2(\calC,T^{\ast}\calC)$ of symmetric traceless two-tensors. With respect to these frames the divergence operator 
$\delta^{g_\ell}\colon S_0^2(\calC,T^{\ast}\calC)\to\Omega^1(\calC)$ takes the form
\begin{equation*}
\delta^{g_\ell}=\begin{pmatrix}\nabla_1+\frac{2\tau}{\sqrt{\tau^2+\ell^2}}&\nabla_2\\
-\nabla_2&\nabla_1+\frac{2\tau}{\sqrt{\tau^2+\ell^2}}\end{pmatrix}.
\end{equation*}
With the Fourier decompositions $\Phi=\sum_{k\in\Z}\varphi_ke^{ik\theta}$,  $\Psi=\sum_{k\in\Z}\psi_ke^{ik\theta}$ we obtain that $\delta^{g_\ell}=\oplus_{k\in\Z}\delta_{\ell,k}$. The  same change of basis $\rho_1 = \sigma_1 - i \sigma_2$, $\rho_2 = \sigma_1 + i \sigma_2$ as before 
diagonalizes each $\delta_{\ell,k}$,  i.e.
\begin{equation*}
\delta_{\ell,k}=\begin{pmatrix}\sqrt{\tau^2+\ell^2}\frac{d}{d\tau}+\frac{2\tau-k}{\sqrt{\tau^2+\ell^2}}&0\\0&\sqrt{\tau^2+\ell^2}\frac{d}{d\tau}+\frac{2\tau+k}{\sqrt{\tau^2+\ell^2}}\end{pmatrix}.
\end{equation*}
Solutions of the homogeneous equation $\delta_{\ell,k}(\lambda_k,\mu_k)^T=0$ are given by
\begin{equation}\label{eq:transfreebasis}
(\lambda_k,\mu_k)=\left(\frac{e^{\frac{k}{\ell}\arctan(\frac{\tau}{\ell})}}{\tau^2+\ell^2},\frac{e^{-\frac{k}{\ell}\arctan(\frac{\tau}{\ell})}}{\tau^2+\ell^2}\right).
\end{equation}

For the following it is convenient to transform    back to the original basis  and to introduce a normalization.

\begin{prop}\label{prop:tttensors}
Let $\ell>0$. A basis of the Hilbert space  $L^2\mathcal S_{\operatorname{tt}}^{g_{\ell}}$ of square-integrable transverse-traceless 
tensors is given by 
\begin{equation}\label{eq:tttensors}
\kappa_{\ell,0}= \frac{\ell^{\frac32} }{(\arctan\frac{1}{\ell})^{\frac12}} \left( \frac{1}{\ell^2+\tau^2},0\right),\quad 
\nu_{\ell,0}=\frac{\ell^{\frac32}}{(\arctan\frac{1}{\ell})^{\frac12} } \left(0,\frac{1}{\ell^2+\tau^2}\right),
\end{equation}
and 
\begin{align*}
&\kappa_{\ell,k}=\frac{C_{\ell,k}}{\tau^2+\ell^2}\left(\cos(k\theta)\cosh\Big(\frac{k}{\ell}\arctan\Big(\frac{\tau}{\ell}\Big)\Big),-\sin(k\theta)\sinh\Big(\frac{k}{\ell}\arctan\Big(\frac{\tau}{\ell}\Big)\Big)\right),\\
&\nu_{\ell,k}=\frac{C_{\ell,k}}{ \tau^2+\ell^2}\left(\sin(k\theta)\cosh\Big(\frac{k}{\ell}\arctan\Big(\frac{\tau}{\ell}\Big)\Big),\cos(k\theta)\sinh\Big(\frac{k}{\ell}\arctan\Big(\frac{\tau}{\ell}\Big)\Big)\right),
\end{align*}
where $k\geq1$ and 
\begin{equation*}
C_{\ell,k}\coloneqq \sqrt{k}e^{-\frac{k}{\ell}\arctan(\frac{1}{\ell})}. 
\end{equation*}
These constants are chosen to normalize the tensors $\kappa_{\ell,k}$ and $\nu_{\ell,k}$ in the sense that
\begin{equation}\label{eq:tensorsnormalized}
c_0\leq\|\kappa_{\ell,k}\|_{L^2(\calC,g_{\ell})},\|\nu_{\ell,k}\|_{L^2(\calC,g_{\ell})}\leq c_1
\end{equation}
where the constants $c_0, c_1 > 0$ are independent of $k$ and $\ell$.
\end{prop}
\begin{proof}
The tensors $\kappa_{\ell,k}$ and $\nu_{\ell,k}$ are obtained from those in \eqref{eq:transfreebasis} by applying the linear transformation $T$. Hence they are contained in $\mathcal S_{\operatorname{tt}}^{g_{\ell}}$ and form a basis of the Hilbert space $L^2$. To verify inequality \eqref{eq:tensorsnormalized} for $k\geq1$ we compute 
\begin{eqnarray}\label{eq:L2kappa}
\begin{split}
  \left\|\frac{\kappa_{\ell,k}}{C_{\ell,k}}\right\|_{L^2(\calC,g_{\ell})}^2&=&\int_{-1}^1\int_0^{2\pi}\frac{1}{(\tau^2+\ell^2)^2}\Big(\cos^2(k\theta)\cosh^2\big(\frac{k}{\ell}\arctan\big(\frac{\tau}{\ell}\big)\big)\\
&&+\sin^2(k\theta)\sinh^2\big(\frac{k}{\ell}\arctan\big(\frac{\tau}{\ell}\big)\big)\Big)\,d\theta\,d\tau\\
 &=&2\pi\int_0^1\frac{\cosh(\frac{2k}{\ell}\arctan(\frac{\tau}{\ell})) }{(\tau^2+\ell^2)^2}\,d\tau\\
 &=&\frac{2\pi}{\ell^3}\int_0^{\frac{1}{\ell}}\frac{\cosh(\frac{2k}{\ell}\arctan(T)) }{(T^2+1)^2}\,dT\\
&=& \frac{\pi}{2k(k^2+\ell^2)(1+\ell^2 )}\Big(2k\cosh(\frac{2k}{\ell}\arctan(\frac{1}{\ell}))\\
&&+(2k^2+1+\ell^2)\sinh(\frac{2k}{\ell}\arctan(\frac{1}{\ell}))\Big).
\end{split}
\end{eqnarray}
In the limit $\ell\searrow0$, the last expression behaves asymptotically as
\begin{equation*}
C k^{-1} e^{\frac{2k}{\ell}\arctan(\frac{1}{\ell})},
\end{equation*}
and thus after multiplying by $C_{\ell,k}^2$ is uniformly bounded above and below. The corresponding statements for 
$\nu_{\ell,k}$ and in the case $k=0$ follow similarly. 
\end{proof}

We now consider the limiting behavior of the tensors $\kappa_{\ell,k}$ and $\nu_{\ell,k}$ as $\ell\searrow0$. 
We again discuss the cases $k=0$ and $k\geq1$ separately.

\subsubsection*{Case $k=0$.}
Looking at the explicit formul\ae\ in Proposition \ref{prop:tttensors}, we see that both
$\kappa_{\ell,0}(\tau,\theta)$ and $\nu_{\ell,0}(\tau,\theta)$ concentrate along $\tau = 0$ as $\ell \searrow 0$. 
To understand this behavior, we lift these tensors to $\wh{C}$. Thus we set $T=\frac{\tau}{\ell}$, and calculate that
\begin{align}\label{eq:ttzeromode}
\begin{split}
\ell^{\frac{1}{2}}\kappa_{\ell,0}=  \frac{dT^2}{\arctan(\frac{1}{\ell})^{\frac{1}{2}}(1+T^2)^2}-\frac{\ell^2}{\arctan(\frac{1}{\ell})^{\frac{1}{2}}}\,
d\theta^2,\\ 
\ell^{\frac{1}{2}}\nu_{\ell,0}=\frac{\ell}{\arctan(\frac{1}{\ell})^{\frac{1}{2}}(1+T^2)}(dT\otimes d\theta+d\theta\otimes dT).
\end{split}
\end{align}
Both  families of tensors here are normalized in $L^2$ with respect to the area form $\ell\,dT\wedge d\theta$; they converge 
uniformly as $\ell \searrow 0$ to $(\frac{2}{\pi})^{\frac{1}{2}}\frac{dT^2}{(1+T^2)^2}$ and to $0$, respectively.

\subsubsection*{Case $k\geq1$.}
By contrast, the  tensors $\kappa_{\ell,k} $ and $ \nu_{\ell,k}$, $k \geq 1$, in Proposition \ref{prop:tttensors} converge 
uniformly on $\calC$.  We study this now.

First consider $\kappa_{\ell,k}$; we restrict to $\tau>0$ since $\tau<0$ is similar. By definition
\[
\kappa_{\ell,k}(\tau,\theta)=\frac{\sqrt{k}}{\tau^2+\ell^2}\Big(\cos(k\theta)\frac{\cosh(\frac{k}{\ell}\arctan(\frac{\tau}{\ell}))}{e^{\frac{k}{\ell}\arctan(\frac{1}{\ell})}},-\sin(k\theta)\frac{\sinh(\frac{k}{\ell}\arctan(\frac{\tau}{\ell}))}{e^{\frac{k}{\ell}\arctan(\frac{1}{\ell})}}\Big).
\]
Since 
\begin{equation*}
\lim_{\ell\to0}\frac{\arctan(\frac{\tau}{\ell})-\arctan(\frac{1}{\ell})}{\ell}=1-\frac{1}{\tau},
\end{equation*}
it follows that
\begin{eqnarray*}
\lim_{\ell\to0}\frac{\cosh(\frac{k}{\ell}\arctan(\frac{\tau}{\ell}))}{e^{\frac{k}{\ell}\arctan(\frac{1}{\ell})}}&=&\frac{1}{2}\lim_{\ell\to0}\frac{e^{\frac{k}{\ell}\arctan(\frac{\tau}{\ell})}}{e^{\frac{k}{\ell}\arctan(\frac{1}{\ell})}}+\frac{1}{2}\lim_{\ell\to0} \frac{e^{-\frac{k}{\ell}\arctan(\frac{\tau}{\ell})}}{e^{\frac{k}{\ell}\arctan(\frac{1}{\ell})}}\\
&=&\frac{1}{2}e^{(1-\frac{1}{\tau})k},
\end{eqnarray*}
and similarly, 
\begin{equation*}
\lim_{\ell\to0}\frac{\sinh(\frac{k}{\ell}\arctan(\frac{\tau}{\ell}))}{e^{\frac{k}{\ell}\arctan(\frac{1}{\ell})}}=\frac{1}{2}e^{(1-\frac{1}{\tau})k}.
\end{equation*}

We obtain, finally, that for $|\tau|\leq1$, 
\begin{equation}\label{eq:decaytt}
\kappa_{\ell,k}(\tau,\theta) \to \kappa_{0,k}(\tau,\theta)=\frac{\sqrt{k}}{2\tau^2}e^{(1-\frac{1}{|\tau|})k}(\cos(k\theta),
-\operatorname{sgn}(\tau)\sin(k\theta)),
\end{equation}
and similarly, 
\begin{equation}\label{eq:decaytt1}
\nu_{\ell,k}(\tau,\theta)\to \nu_{0,k}(\tau,\theta)=\frac{\sqrt{k}}{2\tau^2}e^{(1-\frac{1}{|\tau|})k}(\sin(k\theta),\operatorname{sgn}(\tau)\cos(k\theta)).
\end{equation} 
Furthermore, the $L^2$ norms of these limits, with respect to the limiting metric $g_0$, are uniformly bounded in $k$.
It can be verified by straightforward calculation that if $k\neq0$, then $\kappa_{0,k}$ and $\nu_{0,k}$ belong to 
$\mathcal S_{\mathrm{tt}}(g_0)$. 

We point out, finally, that the equation $\delta^{g_0}\mu_k=0$ admits the further family of solutions
\begin{align*}
\mu_k(\tau,\theta)=\frac{1}{\tau^2}e^{\frac{k}{|\tau|}}(\cos(k\theta),\operatorname{sgn}(\tau)\sin(k\theta)),\\
\lambda_k(\tau,\theta)=\frac{1}{\tau^2}e^{\frac{k}{|\tau|}}(-\sin(k\theta),\operatorname{sgn}(\tau)\cos(k\theta))\qquad(k\in\mathbb N).
\end{align*}
However, these blow up exponentially at $\tau=0$, and do not enter our considerations further.  

\begin{rem}\upshape
The $\mathrm{tt}$ tensors $\kappa_{0,k}$ and $\nu_{0,k}$ admit a geometric interpretation when $k=0$ and $1$. 
The case $k=0$ corresponds to an infinitesimal change of length and  Dehn twist coordinates, while for $k=1$ these 
tensors represent infinitesimal translations of the node $\{p\}$ along the punctured surface $\Sigma\setminus\{p\}$. 
Passing to local holomorphic coordinates and identifying transverse-traceless tensors with meromorphic quadratic 
differentials, it is not hard to see that $k=0$ corresponds to meromorphic quadratic differentials with poles of order 
$2$, while the case $k=1$ corresponds to those with poles of order $1$.
\end{rem}



\section{Parametrix construction}\label{sect:parametrix}
We now construct a parametrix for the operator $P^{g_{\ell}}$ by gluing together two local parametrices $G_{\ell,0}$ and $G_{\ell,1}$, 
the first defined on some long cylinder $(\mathcal C,g_{\ell})$ and the second on the `thick' part $(\Sigma\setminus 
\mathcal C,g_{\ell})$ of the Riemann surface $(\Sigma,g_{\ell})$. To simplify notation, we carry this out in the case of a single divisor, 
i.e.~when $|J|=1$ in the notation introduced in \textsection \ref{univbundle}. The general case is proved in exactly the same way. 

Using the local coordinates $(\tau,\theta)$ as in \textsection \ref{univbundle}, set $\Sigma_0\coloneqq\Sigma\setminus
\big([-\frac{1}{2},\frac{1}{2}]\times S^1\big)$ and $\mathcal C\coloneqq(-\frac{3}{4}, \frac{3}{4})\times S^1$; we sometimes 
refer to this last set as $\Sigma_1$. Thus $\{\Sigma_0, \Sigma_1\}$ is an open cover of $\Sigma$. Choose a partition of unity 
$\{\chi_0,\chi_1\}$ subordinate to this cover; also, choose functions $\tilde\chi_j\in \calC^{\infty}(\Sigma_j)$, such that 
$\tilde\chi_j\chi_j= \chi_j$. Let $G_{\ell,0}$ and $G_{\ell,1}$ denote the exact inverses of the operator $P^{g_{\ell}}$, say with
Dirichlet conditions at the boundaries, acting on $1$-forms on $\Sigma_0$ and $\Sigma_1$. 
Now define the parametrix
\begin{equation}\label{eq:firstparametrix}
\tilde G_{\ell}= \tilde\chi_0G_{\ell,0}\chi_0+\tilde\chi_1G_{\ell,1}\chi_1.
\end{equation}
It is immediate that $P^{g_{\ell}}\tilde G_{\ell}=\operatorname{Id}+\sum_{j=0,1}[P^{g_{\ell}},\tilde\chi_j]G_{\ell,j}\chi_j$. 

The error term 
\begin{equation*}
R_{\ell}\coloneqq-\sum_{j=0,1}[P^{g_{\ell}},\tilde\chi_j]G_{\ell,j}\chi_j
\end{equation*}
is smoothing. 
Indeed, the supports of $[P^{g_{\ell}},\tilde\chi_j]$ and $\chi_j$ are disjoint and $G_{\ell,j}$ is a pseudodifferential operator, so its Schwartz
kernel is singular only along the diagonal, so the Schwartz kernel of $R_\ell$ is $\calC^\infty$ and has support disjoint from the
diagonal. 

We adjust the parametrix $\tilde G_{\ell}$ slightly to make $R_{\ell}$ vanishes at $\ell=0$. To do this, let
$F_0(z, z') \in \calC^{\infty}(\Sigma\times\Sigma)$ be the solution of
\begin{equation*}
P^{g_0}F_0(z, z') =R_0 (z, z'),
\end{equation*}
where $z'$ is regarded as a parameter. If we restrict $z$ to lie in $|\tau| \leq 1$, then $F_0$ decomposes as $ F_0^0 + F_0^\perp$, 
where the first term is the zero Fourier mode (in $z$) and the other is the sum of all the other Fourier modes. 
Our explicit calculations above show that $F_0^\perp$ vanishes to all orders at $\tau = 0$, while $F_0^0$ is polyhomogeneous there.
Now extend $F_0$ to a polyhomogeneous family $F_{\ell}$ on $[ \Sigma \times \Sigma; \{\tau = \ell=0\}]$. Note that we can assume
that this vanishes for $|\tau' | \leq c$ and for all $\ell$ since $R_0$ vanishes in this region. 

Next define
\begin{equation*}
\bar G_{\ell}\coloneqq\tilde G_{\ell}+F_{\ell}\qquad\textrm{and}\qquad S_{\ell}\coloneqq R_{\ell}-P^{g_{\ell}}F_{\ell}.
\end{equation*}
By construction, $S_{\ell}$ vanishes along the face $\ell=0$, is polyhomogeneous, and its Schwartz kernel $\mathcal S_{\ell}(z,z')$ 
has support in $(\Sigma\setminus \{|\tau| \leq c\})\times\Sigma$.  Observe finally that the family of operators $S_{\ell}$ is uniformly 
bounded on $L^2(\Sigma,dA_{g_{\ell}})$ and converges to $0$ as $\ell \searrow 0$ in the operator norm topology. Since
\begin{equation*}
P^{g_{\ell}}\bar G_{\ell}=P^{g_{\ell}}\tilde G_{\ell}+P^{g_{\ell}}F_{\ell}=\operatorname{id}+
\sum_{j=0,1}[P^{g_{\ell}},\tilde\chi_j]G_{\ell,j}\chi_j+P^{g_{\ell}}F_{\ell}=\operatorname{id}-S_{\ell},
\end{equation*}
we see that for $\ell$ sufficiently small, the exact inverse of $P^{g_{\ell}}$ is given by 
\begin{equation}\label{eq:exactinverse}
G^{g_{\ell}}=\bar G_{\ell}(\operatorname{Id}-S_{\ell})^{-1}\colon L^2(\Sigma,dA_{g_{\ell}})\to L^2(\Sigma,dA_{g_{\ell}}).
\end{equation}
where 
\begin{equation*}
(\operatorname{Id}-S_{\ell})^{-1}=\sum_{k=0}^{\infty}S_{\ell}^k
\end{equation*}
is the norm-convergent Neumann series. 

\begin{lemma}\label{lem:exactinverse}
Suppose that $h_{\ell} $ is polyhomogeneous on $\Sigma \times [0,\ell_0)$ and vanishes for $|\tau|\leq c<1$. Then the 
unique solution to $P^{g_{\ell}}\omega_{\ell}=h_{\ell}$ is polyhomogeneous on $[ \Sigma \times [0,\ell_0); \{\tau = \ell=0\}]$
\end{lemma}
\begin{proof} 
The solution $\omega_\ell$ equals $G^{g_\ell} h_\ell = \bar G_\ell k_\ell$ where $k_\ell = (\mbox{Id} - S_\ell)^{-1} h_\ell$, or equivalently,
$k_\ell = h_\ell + S_\ell k_\ell$.  Notice that both terms on the right vanish for $|\tau| \leq c$, the first term by hypothesis
and the second because $S_\ell(z,z')$ vanishes in this region. Therefore $k_\ell$ itself vanishes near $\tau = 0$.
Therefore, by Proposition~\ref{eq:inversecylinder}, $\omega_\ell = \bar G_\ell k_\ell$ is polyhomogeneous, as claimed.
\end{proof}

\section{Proof of the main result}\label{sect:ttsubbundle}
Following the notation of \textsection \ref{univbundle}, fix $q_0 \in H_j$ and let $\calV \subset \Mgt$ be  a neighborhood 
of $q_0$. 
Recall also from Proposition \ref{prop:verticaltwotensor} the existence of a polyhomogeneous slice, i.e., family of 
symmetric $2$-tensors $h_q$ on the vertical tangent bundle $T^{\ver} \Mgonet$, which restricts to a family of 
approximately hyperbolic metrics. 

In this final section, we complete the main step of our main result. This is done by constructing a local frame for 
the subbundle $\mathcal S_{\operatorname{tt}}(h_q)\subseteq\operatorname{Sym}^2(T^{\ver} \Mgonet)$ of 
transverse-traceless two-tensors which depends in a polyhomogeneous way on the Fenchel-Nielsen coordinates on $\calV$. 
We first construct 
a polyhomogeneous $(6\gamma-6)$-frame whose elements are approximately transverse-traceless, and then 
correct these sections to be exactly transverse-traceless, preserving polyhomogeneity in the process.   This construction 
relies crucially on the polyhomogeneity of solutions $\omega_{\ell}$ to $P^{g_{\ell}}\omega_{\ell} =h_{\ell}$ when $h_{\ell}$ 
has support disjoint from the set of degenerating central geodesics, cf.~Lemma \ref{lem:exactinverse}.  The fact
that these sections remain independent is because the correction terms are uniformly small.

\begin{lemma}\label{lem:basicprojestimate}
Let $g$ be a smooth metric on $\Sigma$ and $G$ the unique hyperbolic metric conformal to $g$, i.e., $G=e^{2u}g$ for 
some $u\in \calC^{\infty}(\Sigma)$.  Then there exists a constant $C>0$, which only depends on $\|u\|_{\calC^0}$, such that 
\begin{equation*}
\|T^g\kappa-\kappa\|_{L^2(\Sigma, dA_g)}\leq C \|\delta^g\kappa\|_{L^2(\Sigma, dA_g)}
\end{equation*}
for all $\kappa\in S_0^2(\Sigma, T^*\Sigma)$. 
\end{lemma}
\begin{proof}
We first prove this estimate when $g = G$ is already hyperbolic. Set $\sigma=T^g\kappa-\kappa$. 
Note that $B^g\kappa=\delta^g\kappa$ since $\tr^g \kappa = 0$; furthermore, by \eqref{defT}, we have
$\sigma=-\pi^g(\delta^g)^{\ast}G^g \delta^g\kappa$. The claim in this case then follows from the estimate 
\begin{eqnarray}\label{eq:basicprojestimate}
\begin{split}
\|\sigma\|_{L^2(\Sigma,dA_g)}^2&=&\langle\pi^g(\delta^g)^{\ast}G^g \delta^g\kappa,\pi^g(\delta^g)^{\ast}G^g \delta^g\kappa\rangle_g\\
&=& \langle\pi^g(\delta^g)^{\ast}G^g \delta^g\kappa, (\delta^g)^{\ast}G^g \delta^g\kappa\rangle_g\\
&=&\langle \delta^g\pi^g(\delta^g)^{\ast}G^g \delta^g\kappa, G^g \delta^g\kappa\rangle_g\\
&=&\langle B^g(\delta^g)^{\ast}G^g \delta^g\kappa, G^g\delta^g\kappa\rangle_g\\
&=&\langle P^g G^g \delta^g\kappa, G^g\delta^g\kappa\rangle_g\\
&=&\langle \delta^g\kappa, G^g \delta^g\kappa\rangle_g\\
&\leq& \|\delta^g\kappa\|_{L^2(\Sigma,dA_g)}^2.
\end{split}
\end{eqnarray}
The fourth identity again uses that $B^g=\delta^g$ on trace-free tensors. The last inequality holds because $P^g \geq 1$ (as
a self-adjoint operator on $L^2$), cf.~\eqref{eq:Weitzenboeck}, since $K^g=-1$. This proves the claim in the case where $g$ is hyperbolic. 

Consider now the case of a general metric $g$, where $G = e^{2u} g$.  We have $\sigma=T^g\kappa-\kappa$ as before,
but now write $\sigma_1=T^G\kappa-\kappa$. Since $\mathcal S_{\operatorname{tt}}(g) = \mathcal S_{\operatorname{tt}}(G)$,
and $T^g$ is an orthogonal projector, it follows that
\begin{equation*}
\|\sigma\|_{L^2(\Sigma, dA_g)}\leq \|\sigma_1\|_{L^2(\Sigma, dA_g)}.
\end{equation*}
Using the general identity $\delta^Gh=e^{-2u} \delta^gh$ for traceless tensors $h$, and taking norms with respect to $g$, 
we can  further estimate
\begin{multline*}
\|\sigma\|_{L^2(\Sigma, dA_g)}\leq \|\sigma_1\|_{L^2(\Sigma, dA_g)}\leq C\|\sigma_1\|_{L^2(\Sigma, dA_G)}\\
\leq C\|\delta^G\kappa\|_{L^2(\Sigma, dA_G)}\leq C\|\delta^g\kappa\|_{L^2(\Sigma, dA_g)},
\end{multline*}
where the constant $C>0$ depends only on $\|u\|_{C^0(\Sigma)}$. Here the third inequality holds by \eqref{eq:basicprojestimate}. This proves the claim in the general case.
\end{proof}

We shall apply Lemma \ref{lem:basicprojestimate} to the family $g_{\ell}$ of approximately hyperbolic metrics. To do so,
it is clearly important to show that the constant $C$ appearing in this Lemma is uniform in $\ell$ as $\ell \searrow 0$. 
In other words, we must prove that the conformal factor is uniformly bounded. 
\begin{lemma}\label{lem:boundconffactor}
Let $g_{\ell}$ be the family of approximately hyperbolic metrics and let $G_{\ell}=e^{2u_{\ell}}g_{\ell}$ be the unique hyperbolic 
metric conformally equivalent to $g_{\ell}$. Then there are constants $c>0$ and $\ell_0>0$ such that  
\begin{eqnarray*}
\|u_{\ell}\|_{C^0(\Sigma)}\leq c 
\end{eqnarray*}
for $0\leq\ell\leq\ell_0$.
\end{lemma}
\begin{proof}
By construction, the metric $g_0$ is hyperbolic, and $g_{\ell}\to g_0$ uniformly in $\calC^\infty$ on $\Sigma\setminus A$, where 
$A$ is the annulus $(-\tau_0,\tau_0)\times S^1$, as $\ell\searrow0$. Thus $K_{g_{\ell}}<0$ for $\ell$ small enough.  
With $\Delta_{G_{\ell}}$ the (negative semidefinite) scalar Laplacian, $u_{\ell}$ satisfies 
\begin{eqnarray}\label{eq:conffactor}
\Delta_{G_{\ell}}u_{\ell}+K_{G_{\ell}}-e^{-2u_{\ell}}K_{g_{\ell}}=0.
\end{eqnarray}
Since $K_{G_{\ell}}=-1$ and $K_{g_{\ell}}<0$, there are sub- and supersolutions
\begin{eqnarray*}
u_{\operatorname{sub}}\equiv-c,\qquad u_{\operatorname{sup}}\equiv c,
\end{eqnarray*}
for some $c \gg 0$, and for all $\ell < \ell_0$. This means that $-c\leq u_{\ell}\leq c$ for all such $\ell$, as claimed.
\end{proof}
 
Our discussion now splits naturally into two cases. 
\subsubsection*{Transverse-traceless tensors concentrating at a central geodesic}
For the rest of this section, we fix the following notation. Consider a neighborhood $\calV =[0,\varepsilon)_{\ell}\times S^1\times\calW$
where $\calW$ is a neighborhood in some divisor $D_j$. Let $q\in\calV$. Then we use the notation $g_{\ell}$ for a hyperbolic metric on $\Sigma$ (which is degenerate if $\ell=0$) representing the point $q$. We suppress its dependence on the remaining 
Fenchel-Nielsen coordinates.

Let $\kappa_{\ell,0}$ and $\nu_{\ell,0}\in\mathcal S_{\operatorname{tt}}(g_{\ell})$ be the symmetric transverse-traceless $2$-tensors 
\eqref{eq:tttensors} on  the cylinder $(\calC,g_{\ell})$.  Fix a smooth cutoff function $\chi\colon[-1,1]\to[0,1]$ with 
$\operatorname{supp}(\chi)\subseteq[-\frac{3}{4},\frac{3}{4}]$ and $\chi\equiv1$ for $|\tau| \leq \frac{1}{2}$. Now set 
\begin{equation}\label{eq:muell}
\hat\mu_{\ell}^1\coloneqq\chi\kappa_{\ell,0}\qquad\textrm{and}\qquad\hat\mu_{\ell}^2\coloneqq\chi\nu_{\ell,0},
\end{equation}
which we extend by $0$ to all of $\Sigma$, and then consider their projections 
\begin{equation}\label{eq:muell1}
\mu_{\ell}^1\coloneqq T^{g_{\ell}}\hat\mu_{\ell}^1\qquad\textrm{and}\qquad\mu_{\ell}^2\coloneqq T^{g_{\ell}}\hat\mu_{\ell}^2
\end{equation}
to $\Stt(g_{\ell})$. 

\begin{prop}
The families $\mu_{\ell}^1$, $\mu_\ell^2$ are polyhomogeneous. 
\end{prop}
\begin{proof}
Since the tensors $\kappa_{\ell,0}$ and $\nu_{\ell,0}$ are divergence-free with respect to $g_{\ell}$, then certainly
$\delta^{g_\ell} \hat\mu_{\ell}^j$ vanishes except when $\frac{1}{2}\leq|\tau|\leq\frac{3}{4}$. Thus Lemma \ref{lem:exactinverse} 
can be applied, with $h_{\ell}=\delta^{g_{\ell}}\hat\mu_{\ell}^j$, and shows that the family of $1$-forms 
$G^{g_{\ell}}\delta^{g_{\ell}}\hat\mu_{\ell}^j$ is polyhomogenous. The claim then follows from \eqref{defT}. 
\end{proof}

We must also prove that $\mu_{\ell}^1$ and $\mu_{\ell}^2$ are linearly independent when $\ell$ is small. This is proved 
in the remainder of this section. By Lemma \ref{lem:basicprojestimate}, it remains to estimate the divergences 
of $\hat\mu_{\ell}^1$ and $\hat\mu_{\ell}^2$.

\begin{prop}\label{prop:extensioncyl}
The divergence of $\hat\mu_{\ell}^j$ vanishes outside $A' :=\{\frac{1}{2}\leq|\tau|\leq \frac{3}{4}\}$ and satisfies  
\begin{equation*}
\|\delta^{g_{\ell}}\hat\mu_{\ell}^j\|_{L^2(\Sigma,dA_{g_{\ell}})}\leq C\ell^{\frac{3}{2}}\qquad(j=1,2).
\end{equation*}
Moreover, $||\sigma_{\ell}^j||_{L^2}=||\mu_{\ell}^j-\hat\mu_{\ell}^j||_{L^2} \to 0$ as $\ell\searrow0$.
\end{prop}
\begin{proof}
Since $\kappa_{\ell,0}$ is divergence-free, it follows from \eqref{eq:tttensors} that
\begin{equation*}
\delta^{g_{\ell}}\hat\mu_{\ell}^1= (\partial_{\tau}\chi) \sqrt{\tau^2+\ell^2}\frac{\ell^{\frac{3}{2}}}{\arctan(\frac{1}{\ell})^{\frac{1}{2}}(\ell^2
+\tau^2)}\sigma_1,
\end{equation*}
which has support in $A'$.  We estimate 
\begin{eqnarray*}
\|\delta^{g_{\ell}}\hat\mu_{\ell}^1\|_{L^2(\Sigma,dA_{g_{\ell}})}^2&=&2\int_{\frac{1}{2}}^{\frac{3}{4}}\int_0^{2\pi} \frac{\ell^3|\partial_{\tau}\chi(\tau)|^2}{\arctan(\frac{1}{\ell})(\tau^2+\ell^2)}\,d\theta\,d\tau \\
&\leq&\frac{C\ell^3}{\arctan(\frac{1}{\ell})}\int_{\frac{1}{2}}^{\frac{3}{4}}\frac{d\tau}{\tau^2+\ell^2}\\
&=&\frac{C\ell^2}{\arctan(\frac{1}{\ell})}\Big(\arctan(\frac{3}{4\ell})-\arctan(\frac{1}{2\ell})\Big). 
\end{eqnarray*}
Observing that
\begin{equation*}
\lim_{\ell\to0}\frac{\arctan(\frac{3}{4\ell})-\arctan(\frac{1}{2\ell})}{\ell}=\frac{2}{3},
\end{equation*}
the assertion on the decay rates of $\delta^{g_{\ell}}\hat\mu_{\ell}^j$, $j = 1,2$, is immediate. 

For the second claim, let $G_{\ell}=e^{-2u_{\ell}}g_{\ell}$ be the hyperbolic metric, as before. By Lemma \ref{lem:boundconffactor}, 
the family of conformal  exponents $u_{\ell}$ is $\calC^0$ bounded, so we can apply Lemma \ref{lem:basicprojestimate}
to get the result.   
\end{proof}

\subsubsection*{Transverse-traceless tensors decaying on long cylinders} 
We next construct a local $(6\gamma-8)$-frame of transverse-traceless tensors which is polyhomogeneous  and complements
the two sections determined in the last subsection.  Taking all these transverse-traceless tensors together, we will have
obtained a local polyhomogeneous frame of the rank-$(6\gamma-6)$ bundle $\mathcal S_{\operatorname{tt}}(g_q)$. 

Fix $q_0\in H_j$ and let $(\Sigma,g_0)$ be the complete surface representing $q_0$. The space of symmetric tensors 
which are transverse-traceless with respect to $g_0$ and decay along the cusps is denoted $\Stt(g_0)$. It is well-known
that the real dimension of this space is $6\gamma-8$.  Fix an orthonormal basis $\{\kappa^3,\ldots,\kappa^{6\gamma-6}\}$ 
for this space. By \eqref{eq:decaytt}, each $\kappa^j$ decays exponentially as $\tau\to 0$, so we obtain the approximately 
orthonormal local frame
\begin{equation}\label{eq:approxframe}
\mu^j(q)\coloneqq\chi\kappa^j\qquad(q\in\calV, j=3,\ldots,6\gamma-6),
\end{equation}
which vanishes outside the cylinder, is identically $1$ near $\tau = 0$, and so $\del_\tau \chi$ is supported in
some region $A_1' = \{0 < \tau_1 \leq \tau \leq \tau_2 \}$.  The constants $\tau_j$ will be fixed later. We also 
write $\calC_1\coloneqq A_1' \times S^1$ and  $\calC_2\coloneqq \{|\tau| \leq \tau_1 \}\times S^1$. 
We now define 
\begin{equation}\label{eq:correctiontt}
\mu_{\ell}^j\coloneqq T^{g_{\ell}}\mu^j, \qquad j=3,\ldots,6\gamma-6,
\end{equation}
i.e.,~$\mu_{\ell}^j$ is the orthogonal projection of $\mu^j$ onto its transverse-traceless part with respect to the metric $g_{\ell}$.  

\begin{prop}\label{prop:framethickpart}
The projected family $\{\mu_{\ell}^j\}$ is polyhomogeneous and for $\ell$ sufficiently small, these vectors are independent.
\end{prop}

\begin{proof}
The first statement follows immediately from the polyhomogeneity of $(\delta^{g_{\ell}})^{\ast}$ and of the family of $1$-forms
$G^{g_{\ell}}\delta^{g_{\ell}}\mu^j$. Note that we use here that $\mu^j$ is supported away from $\{\tau=0\}$. 

To prove independence, we consider
\begin{equation*}
\delta^{g_{\ell}}\mu^j=\delta^{g_{\ell}}(\chi\kappa^j)=\chi\delta^{g_{\ell}}\kappa^j+\partial_{\tau}\chi\sqrt{\tau^2+\ell^2}\kappa^j(E_1)
\end{equation*}
Since $\operatorname{supp}\chi \subseteq\Sigma\setminus \calC_2$ and $\operatorname{supp}\partial_{\tau}\chi\subseteq \calC_1$, 
it follows that
\begin{multline}\label{eq:divmu_j}
\frac{1}{2}\|\delta^{g_{\ell}}\mu^j\|_{L^2(\Sigma,dA_{g_{\ell}})}^2\leq\\
\|\chi\delta^{g_{\ell}}\kappa^j\|_{L^2(\Sigma\setminus \calC_2,dA_{g_{\ell}})}^2+\|\partial_{\tau}\chi\sqrt{\tau^2+\ell^2}\kappa^j(E_1)\|_{L^2(\calC_1,dA_{g_{\ell}})}^2.
\end{multline}
Fix $\varepsilon>0$. Choosing $\tau_2 \ll 1$ and $\ell_0>0$ sufficiently small, it follows that if $\ell\leq\ell_0$ and 
$j=3,\ldots,6\gamma-6$, the right-hand side of \eqref{eq:divmu_j} is bounded by $C\varepsilon$. Indeed, on 
$\Sigma\setminus \calC_2$, $\delta^{g_{\ell}}\kappa^j$ converges uniformly to $\delta^{g_{\ell_0}}\kappa^j=0$. As for the 
second term on the right in  \eqref{eq:divmu_j}, the equations \eqref{eq:decaytt} and \eqref{eq:decaytt1} imply that 
$\kappa^j(\tau,\theta)$ decays rapidly as $\tau\to 0$, hence this term is small also. To conclude, 
Lemma \ref{lem:basicprojestimate} give finally that
\begin{equation*}
\|\mu_{\ell}^j-\mu^j\|_{L^2(\Sigma,dA_{g_{\ell}})}=\|T^{g_{\ell}}\mu^j-\mu^j\|_{L^2(\Sigma,dA_{g_{\ell}})}\leq 
C\|\delta^{g_{\ell}}\mu^j\|_{L^2(\Sigma,dA_{g_{\ell}})},
\end{equation*}  
where $C$ does not depend on $\ell$. Hence $\{\mu_{\ell}^3,\ldots,\mu_{\ell}^{6\gamma-6}\}$ is independent when $\varepsilon$ 
is small enough. 
\end{proof}

\subsubsection*{Proof of the main result} 
We have now constructed the frame $\{\mu_{\ell}^1,\ldots,\mu_{\ell}^{6\gamma-6}\}$ of transverse-traceless 
tensors, and the remaining task is to show that these span the vector space $\mathcal S_{\operatorname{tt}}(g_q)$, for each $q\in\calV$.

\begin{lemma}\label{lem:ttsubbundle}
There is a constant $\ell_0>0$ such that for all $0\leq\ell<\ell_0$   the set
\begin{equation}
\{\mu_{\ell}^1,\mu_{\ell}^2,\mu_{\ell}^3,\ldots,\mu_{\ell}^{6\gamma-6}\}
\end{equation}
is a polyhomogeneous  local frame over $\calV$ of the vector bundle $\mathcal S_{\operatorname{tt}}(g_q)$.
\end{lemma}
\begin{proof}
Recall  the symmetric (but not necessarily divergence-free) tensors $\hat\mu_{\ell}^i$, $i=1,2$ These  are supported on the 
cylinder $\calC\subseteq\Sigma$ and their coefficients depend only on the variable $\tau$. We then consider the family 
\eqref{eq:approxframe}. Altogether, this gives another family $\{\mu^1,\ldots,\mu^{6\gamma-6}\}$ of tensors whose restriction 
to $\calC$ has vanishing zeroth Fourier mode. 
\begin{equation*}
\langle \hat\mu_{\ell}^i,\mu^j\rangle_{g_{\ell}}=0
\end{equation*}
for all $i=1,2$, $j=3,\ldots,6\gamma-6$, and any $\ell>0$. By Propositions \ref{prop:extensioncyl} and \ref{prop:framethickpart}  the estimates
\begin{multline*}
\|\mu_{\ell}^i-\hat\mu_{\ell}^i\|_{L^2(\Sigma,dA_{g_{\ell}})}<\varepsilon\quad(i=1,2)\\
\textrm{and}\qquad\|\mu_{\ell}^j -\mu^j \|_{L^2(\Sigma,dA_{g_{\ell}})}<\varepsilon\quad(j=3,\ldots,6\gamma-6)
\end{multline*} 
hold for all $\varepsilon>0$ and every sufficiently small $0\leq\ell\leq\ell_0(\varepsilon)$. This implies that
\begin{eqnarray*}
|\langle\mu_{\ell}^i,\mu_{\ell}^j\rangle_{g_{\ell}}|&\leq&|\langle\mu_{\ell}^i-\hat\mu_{\ell}^i,\mu_{\ell}^j -\mu^j\rangle|+|\langle\mu_{\ell}^i-\hat\mu_{\ell}^i,\mu^j\rangle_{g_{\ell}}|\\
&&+|\langle\hat\mu_{\ell}^i,\mu_{\ell}^j-\mu^j\rangle_{g_{\ell}}|+|\langle\hat\mu_{\ell}^i,\mu^j\rangle_{g_{\ell}}|\\
&<&\varepsilon^2+\varepsilon(\|\hat\mu_{\ell}^i\|_{L^2(\Sigma,dA_{g_{\ell}})}+\|\mu^j\|_{L^2(\Sigma,dA_{g_{\ell}})})\\
&\leq&C\varepsilon.
\end{eqnarray*}
Hence if $\varepsilon$ is sufficiently small, the subspaces spanned by 
$\{\mu_{\ell}^1,\mu_{\ell}^2\}$ and  $\{\mu_{\ell}^3,\ldots,\mu_{\ell}^{6\gamma-6}\}$ are transversal. This implies the claim.
\end{proof}

We are now in position to prove our main result.

\medskip

\noindent {\it Proof of Theorem~\ref{thm:mainthm}} 
It suffices to establish our result in an open set $\calV$ around a point $q_0 \in \del \wtM_\gamma$.  Let us choose a polyhomogeneous
slice of approximately hyperbolic metrics $g(w)$ in $\calV$.  By Lemma  \ref{lem:ttsubbundle} yields the existence of a local 
polyhomogeneous frame $\kappa_1, \ldots, \kappa_{6\gamma-6}$ of sections for the bundle of transverse-traceless tensors over $\calV$.  
In addition, by \cite{MZ},  the family of conformal factors $e^{2\varphi(w)}$ relating the approximately hyperbolic metrics $g(w)$ to the 
exact hyperbolic metrics on each fibre is also polyhomogeneous on $\Mgonet$.   The matrix coefficients of $g_{\WP}$
are given by the expression
\[
(g_{\WP})_{ij} = \int_\Sigma  \langle \kappa_i , \kappa_j \rangle_{g(w)}  e^{-2\varphi(w)} \, dA_{g(w)}.
\]
Everything in this expression is polyhomogeneous.  To finish, we observe that the integral over $\Si$ can be regarded as
a pushforward with respect to the $b$-fibration $\wh{\Pi}: \Mgonet \to \Mgt$. We may therefore invoke the properties
of polyhomogeneous functions with respect to pushforwards by $b$-fibrations, as proved in \cite{MelCCD}. This
theorem proves that each $(g_{\WP})_{ij}$ is polyhomogeneous on $\Mgt$. More precisely, the powers in the polyhomogeneous
expansions of each quantity here are nonnegative integer powers of $\ell^{1/2}$ (or in half-integer powers of the appropriate
boundary defining functions at each face of $\Mgonet$), and each term $\ell^{k/2}$ is possibly multiplied by a polynomial 
in $\log \ell$. The pushforward theorem then asserts that this pushforward has an expansion of exactly the same form. 

As we have noted before, the result of Melrose and Zhu at present only asserts polyhomogeneity of $\varphi$ near $\BbD^{\reg}$,
so our result only applies near this portion of $\Mgonet$. However, they expect their result to hold in general, and when
that is complete, our result here will then extend.  
$\Box$

\end{document}